\documentclass{article}
\usepackage{graphicx} 
\usepackage[a4paper, margin=2.5cm]{geometry}
\usepackage{amsthm} 
\usepackage{amssymb}
\newtheorem{theorem}{Theorem}
\newtheorem{corollary}{Corollary}
\newtheorem{lemma}{Lemma}

\newtheorem{proposition}{Proposition}
\newtheorem{openproblem}{Open Problem}
\newtheorem{definition}{Definition}

\usepackage{amsmath}
\usepackage{float}
\usepackage{subfigure}
\usepackage{tikz}
\usepackage{titlesec}
\usepackage{url}
\titlelabel{\thetitle.\quad}
\usepackage{comment}

\title{On Generalized von Neumann Inverse Graphs of Finite Commutative Regular Rings}
\author{Felicia Servina Djuang$^{1}$, Indah Emilia Wijayanti$^{2}$, and Yeni Susanti$^{3}$\\
{\small $^{1,2,3}$Department of Mathematics, Universitas Gadjah Mada, Yogyakarta, Indonesia}\\
\small{$^{1}$feliciadjuang25@mail.ugm.ac.id, $^{2}$ind\_wijayanti@ugm.ac.id, $^{3}$yeni\_math@ugm.ac.id}}
\date{}

\begin{document}

\maketitle
\begin{abstract}
Let $R$ be a ring with identity. The generalized von Neumann inverse graph of $R$, denoted by $\Gamma_{Reg}(R)$, is defined as the graph whose vertex set is $Reg(R)$, where two distinct vertices $a,b \in R$ are adjacent if and only if $aba=a$ or $bab=b$. In this work, we consider the reduced graph $\Gamma'_{Reg}(R)$ obtained by restricting the vertex set to $Reg(R)\setminus{0_R}$, so that $\Gamma_{Reg}(R) \cong K_1 + \Gamma'_{Reg}(R)$, allowing the analysis to focus on its nontrivial structure. We investigate the structure of $\Gamma'_{Reg}(R)$ for finite commutative von Neumann regular rings and establish several results describing its graph-theoretic properties in relation to the algebraic structure of $R$. In particular, we derive conditions that characterize connectivity, acyclicity, and planarity, and examine structural features such as vertex degrees, girth, and the existence of pendant vertices, along with their algebraic implications. We also identify circumstances under which $\Gamma'_{Reg}(R)$ exhibits specific graph classes, including paths, cycles, and wheels, as well as the presence of certain induced subgraphs. Furthermore, an explicit algorithm is provided to construct $\Gamma'_{Reg}(R)$, and connections with the inclusion ideal graph of $R$ are discussed, offering additional insight into the interplay between ring-theoretic properties and graph structures.
\end{abstract}

\section{Introduction}
Graph theory has developed into a central area of discrete mathematics, providing a flexible language for describing relationships between objects and a wide range of tools for analyzing their structure. In recent years, this perspective has been extended beyond purely combinatorial settings to incorporate algebraic ideas, giving rise to algebraic graphs. In these constructions, graphs are associated with algebraic systems such as groups, rings, or modules, allowing algebraic properties to be examined through combinatorial representations. This connection has proven to be effective, as structural features of algebraic objects can often be translated into graph-theoretic terms, making them easier to interpret and analyze.

A brief review of ring theory is helpful for establishing the context. Let $R$ be a ring with identity. An element $a \in R$ is called \textbf{idempotent} if $a^2=a$, an element $u \in R$ is called a \textbf{unit} if there exists $v \in R$ such that $uv=vu=1$, and an element $x \in R$ is called \textbf{von Neumann regular} if there exists $y \in R$ such that $xyx=x$. A ring $R$ is said to be a \textbf{regular ring} if every element of $R$ is von Neumann regular. The properties of regular rings and von Neumann regular elements have been extensively studied in the literature, as discussed in \cite{goodearl, andersonbadawi, alahmadi}. In this context, the smallest ideal of $R$ containing an element $a \in R$ will be denoted by $[a]$. Moreover, the general structure of finite commutative von Neumann regular rings can be described using classical results found in \cite{wisbauer, herstein, dummit, knapp}.

Algebraic graphs were first introduced through Cayley graphs \cite{cayley}, which are defined on groups. Given a group $G$ and a subset $S \subseteq G$ that does not contain the identity and is closed under taking inverses, the Cayley graph $\mathrm{Cay}(G,S)$ is a graph whose vertex set is $G$, where two vertices $x$ and $y$ are adjacent if and only if $x^{-1}y \in S$ (equivalently, $y = xs$ for some $s \in S$). Here, the set $S$ is commonly referred to as the connection set, and it determines the adjacency structure of the graph. This construction has inspired various generalizations, including several new graph definitions on groups \cite{vasantha}. 

Motivated by this idea, analogous constructions have been extended to rings. One such construction defines a Cayley-type graph on a ring $R$, where two vertices $x$ and $y$ are adjacent if and only if $x - y \in S$, depending on a chosen connection set $S \subseteq R$. When the connection set is taken to be the set of all units in $R$, the resulting graph is known as the unitary Cayley graph \cite{dejter}. A further variation replaces subtraction with addition, where two vertices $x$ and $y$ are adjacent if and only if $x + y \in S$, which is commonly referred to as the addition Cayley graph. If $S$ is the set of all units of the ring, the graph is called the unitary addition Cayley graph \cite{sinha}. When the connection set consists of all von Neumann regular elements of the ring, the graph is known as the von Neumann regular graph \cite{kharkongor}. Moreover, if the connection set is the set of all unit regular elements, the graph is referred to as the unit regular graph \cite{susanti}.

On the other hand, the construction of graphs on rings was initiated by Beck through the study of colorings of zero-divisor graphs \cite{beck}. In Beck's definition, the vertices are all elements of the ring, and two distinct vertices $x$ and $y$ are adjacent if $xy=0$. This definition was later refined by Anderson \cite{anderson} by restricting the vertex set to the set of all zero-divisors of the ring. In addition, the idempotent graph was introduced in \cite{akbari}, where the vertex set consists of all idempotent elements of the ring, and two distinct vertices are adjacent under a similar zero-product condition. 

Further developments include the introduction of the clean graph of ring $R$, denoted by $Cl(R)$ \cite{habibiyet}, whose vertex set consists of ordered pairs of idempotent and unit elements of $R$. Two distinct vertices $(e,u)$ and $(f,v)$ are adjacent if $ef=fe=0$ or $uv=vu=1$. A related graph, denoted by $Cl_2(R)$, is obtained by restricting the vertex set to pairs of nonzero idempotents and units. These constructions are closely related to the notion of clean rings, where every element can be expressed as the sum of an idempotent and a unit. The adjacency condition combines the zero-divisor relation on idempotents with the inverse relation on units. Several numerical parameters of clean graphs of commutative rings have been studied in \cite{singhpatekar}, and various isomorphism results, as well as structural descriptions for clean graphs of certain matrix rings, have been established in \cite{djuang3}. In fact, the general structure of clean graphs has been shown to depend only on the structure of the corresponding idempotent graph \cite{djuang1}, where a $(t,n)$-shuriken operation (with $n-t$ even) is introduced. More precisely, the graph $Cl_2(R)$ is isomorphic to the graph obtained by applying the $(|U'(R)|,|U(R)|)$-shuriken operation to the idempotent graph of $R$, where $U(R)$ denotes the set of all units of $R$ and $U'(R)$ denotes the set of all self-inverse units in $R$. This reduction is possible because the inverse relation among units remains relatively simple.

Beyond clean rings, there is also the notion of $r$-clean rings, in which each element can be written as the sum of an idempotent and a von Neumann regular element \cite{ashrafi}. Motivated by the construction of clean graphs, one can define a related algebraic graph whose vertices consist of pairs of idempotent elements and von Neumann regular elements, with adjacency defined by combining the zero-divisor relation on idempotents and inverse-type relations on regular elements. In this context, it is important to distinguish between two types of inverses: an element $a$ is called an inner inverse of $b$ if $bab=b$, and an outer inverse if $aba=a$ \cite{andersonbadawi}. These inverses need not be unique. To better understand such inverse relations among von Neumann regular elements, a graph-theoretic approach can be employed. 

Accordingly, the generalized von Neumann inverse graph of a ring $R$, denoted by $\Gamma_{Reg}(R)$, is defined as the graph whose vertex set is the set of all von Neumann regular elements of $R$, denoted by $Reg(R)$, where two distinct vertices $a$ and $b$ are adjacent if $aba=a$ or $bab=b$. In this paper, we study the generalized von Neumann inverse graph of finite commutative regular rings.

In addition to the algebraic graphs mentioned above, there is also the inclusion ideal graph of a ring $R$, denoted by $In(R)$. This graph has as its vertex set all nontrivial ideals of $R$, and two distinct vertices $I$ and $J$ are adjacent if $I \subset J$ or $J \subset I$ \cite{akbari2014, akbari2015}. The inclusion ideal graph captures the hierarchical structure of ideals in a ring through inclusion relations. Moreover, this graph is closely related, from a structural point of view, to the generalized von Neumann inverse graph of finite commutative regular rings, providing further insight into the interplay between ideal structure and graph representations.

Recent work has shown that algebraic graphs can play a meaningful role in coding theory. In \cite{fish}, linear codes were constructed from the incidence matrix of the line graph of the Hamming graph, demonstrating how combinatorial structures can be used to generate codes with desirable properties. Building on this idea, \cite{jain} developed linear codes using the incidence matrix of the unit graph over $\mathbb{Z}_n$, where the construction depends on the number of prime factors of $n$. 

In this direction, the generalized von Neumann inverse graph, beyond its role in clarifying and illus\-trating the properties of von Neumann regular elements in a ring, is also expected to provide a new framework for constructing linear codes. Its structure, which reflects inverse-type relations among regular elements, offers potential for further exploration in coding theory.

Before proceeding, we recall several basic graph-theoretic concepts that are used throughout this paper. Let $G=(V(G),E(G))$ be a graph, where $V(G)$ and $E(G)$ denote the vertex set and the edge set of $G$, respectively. A graph $G$ is said to be \textbf{connected} if there exists a path between every pair of distinct vertices. A \textbf{path graph}, denoted by $P_n$, is a graph that is isomorphic to a graph with vertex set $\{v_1,v_2,\dots,v_n\}$ and edge set $\{v_iv_{i+1} \mid 1 \leq i \leq n-1\}$. A \textbf{cycle graph}, denoted by $C_n$ for $n \geq 3$, is a graph that is isomorphic to a graph with vertex set $\{v_1,v_2,\dots,v_n\}$ and edge set $\{v_iv_{i+1} \mid 1 \leq i \leq n-1\} \cup \{v_nv_1\}$. A \textbf{wheel graph}, denoted by $W_n$, is obtained from a cycle graph $C_{n}$ by adding a new vertex that is adjacent to every vertex of the cycle. A \textbf{tree} is a connected graph that contains no cycles; equivalently, a graph $G$ is a tree if it is connected and has exactly $|V(G)|-1$ edges. 

In addition, several specific graphs and graph operations will be introduced and used in this study. One of them is the house graph, which can be described as a graph obtained from a cycle graph $C_4$ by adding an additional vertex that is adjacent to two consecutive vertices of the cycle, forming a shape that resembles a simple house. This graph can be illustrated as follows:
    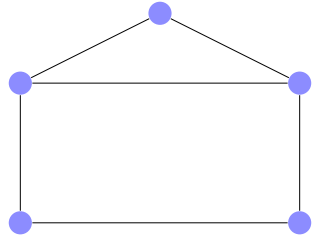
\begin{figure}[H]
	\begin{center} 
		\resizebox{0.25\textwidth}{!}{\begin{tikzpicture}
		[scale=1,auto=center,every node/.style={circle, fill=blue!45}] 
		\node (v5) at (0,0)  {\textit{}};
		\node (v4) at (-2,-1)  {\textit{}};
		\node (v3) at (2,-1)  {\textit{}};
        \node (v2) at (2,-3)  {\textit{}};
        \node (v1) at (-2,-3)  {\textit{}};
		\draw (v1) --  (v2);
		\draw (v1) -- (v4);
		\draw (v2) -- (v3);
        \draw (v3) -- (v4);
        \draw (v3) -- (v5);
        \draw (v4) -- (v5);
		\end{tikzpicture}}
		\caption{House Graph} \label{gambargrafrumah}
	\end{center}
\end{figure}

Let $G$ and $H$ be two graphs such that $V(G) \cap V(H)= \emptyset$. The \textbf{disjoint union} of $G$ and $H$, denoted by $G \cup H$, is defined as the graph with vertex set $V(G) \cup V(H)$ and edge set $E(G) \cup E(H)$. The \textbf{join} of $G$ and $H$, denoted by $G + H$, is the graph with vertex set $V(G) \cup V(H)$ and edge set $E(G) \cup E(H) \cup \{xy \mid x \in V(G), y \in V(H)\}$, where each vertex of $G$ is connected to every vertex of $H$. Let $A \subseteq V(G)$. The graph obtained by deleting the vertices in $A$ from $G$ is denoted by $G - A$, and is defined by
\begin{align*}
    V(G - A) &= V(G) \setminus A \\ 
    E(G - A) &= E(G) \setminus \{ xy \mid x \in A,\ y \in V(G),\ xy \in E(G) \}.
\end{align*}
Furthermore, the canonical double cover of a graph was introduced in \cite{hujdurovic}. The \textbf{canonical double cover} of a graph $G$, denoted by $B(G)$, is defined as the graph with vertex set $V(G) \times \mathbb{Z}_2$ and edge set $\{(x,0)(y,1) \mid xy \in E(G)\}$. In addition, \cite{barragan} introduced the notion of subgraph-amalgamation. Here, we present a more explicit formulation adapted to the needs of this study to ensure that the resulting graph is uniquely determined. Let $\{G_i \mid i \in I\}$ be a family of graphs. Suppose that there exists a subset $J \subseteq V(G_i)$ for each $i \in I$ such that $V(G_i) \cap V(G_j)=J$ for all $i \neq j$, and the induced subgraphs $\langle J \rangle_{G_i}$ are equal for all $i \in I$. The \textbf{subgraph amalgamation} of $\{G_i \mid i \in I\}$ over $J$, or the $J$-amalgamation of $\{G_i \mid i \in I\}$, denoted by $\bigsqcup(G_i \mid J)$, is defined as the graph with vertex-set
$$
V\left(\bigsqcup(G_i \mid J)\right)=\bigcup_{i \in I}(V(G_i)\setminus J) \cup J
$$
and edge-set
$$
E\left(\bigsqcup(G_i \mid J)\right)=\bigcup_{i \in I}(E(G_i)\setminus E(J)) \cup E(J),
$$
where $E(J)$ denotes the edge set of the induced subgraph of $G_i$ for some $i \in I$ on the vertex set $J$.

Moreover, we introduce a new operation that combines the idea of amalgamation \cite{simanjuntak, barragan} with the canonical double cover construction \cite{hujdurovic}, which we call the amalgamation canonical double cover. Let $G=(V(G), E(G))$ be a graph, $J \subseteq V(G)$, and $A \subseteq V(G-J)$. The \textbf{amalgamation canonical double cover} $A$ over $J$ of a graph $G$, denoted by $AmB_A(G \mid J)$, is the graph with $V(AmB_A(G)\mid J)=V(B(G-J)) \cup J$ and
    \begin{align*}
        E(AmB_A(G \mid J))&=E(B(G-J)) \cup \{(v,0)(v,1) \mid v \in A\} \cup E(\langle J \rangle_G) \\
    &\qquad \cup\{u(v,0),u(v,1) \mid u \in J, v \in V(G-J), uv \in E(G)\}.
    \end{align*}
This operation will be used to simplify and clarify the structure of the graphs studied in this paper. To illustrate this construction, consider the graph $C_6$ as shown in Figure \ref{fig:AmB_C6}. Let $J=\{a_1,a_2,a_3\} \subseteq V(C_6)$ and $A=\{b,d\} \subseteq V(C_6-J)$. The following figure presents the resulting graph obtained from the amalgamation canonical double cover.

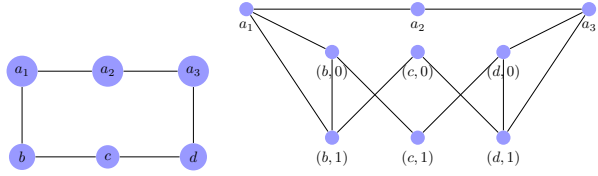
\begin{figure}[H]
\centering
\resizebox{0.5\textwidth}{!}{
\begin{tikzpicture}[scale=1,auto=center,roundnode/.style={circle,fill=blue!40}]
\node[roundnode] (a1) at (-2,2) {$a_1$};
\node[roundnode] (a2) at (0,2) {$a_2$};
\node[roundnode] (a3) at (2,2) {$a_3$};
\node[roundnode] (a4) at (-2,0) {$b$};
\node[roundnode] (a5) at (0,0) {$c$};
\node[roundnode] (a6) at (2,0) {$d$};

\draw (a1)--(a2)--(a3)--(a6)--(a5)--(a4)--(a1);
\end{tikzpicture}\space\space\space\space\space
\begin{tikzpicture}[scale=1,auto=center,roundnode/.style={circle,fill=blue!40}]
\node[roundnode,label=below:{$(b,0)$}] (a1) at (-2,2) {};
\node[roundnode,label=below:{$(c,0)$}] (a2) at (0,2) {};
\node[roundnode,label=below:{$(d,0)$}] (a3) at (2,2) {};

\node[roundnode,label=below:{$(b,1)$}] (b1) at (-2,0) {};
\node[roundnode,label=below:{$(c,1)$}] (b2) at (0,0) {};
\node[roundnode,label=below:{$(d,1)$}] (b3) at (2,0) {};

\node[roundnode,label=below:{$a_1$}] (a4) at (-4,3) {};
\node[roundnode,label=below:{$a_2$}] (a5) at (0,3) {};
\node[roundnode,label=below:{$a_3$}] (a6) at (4,3) {};

\draw (a1)--(b2);
\draw (a2)--(b1);
\draw (a2)--(b3);
\draw (a3)--(b2);
\draw (a4)--(a1);
\draw (a4)--(b1);
\draw (a6)--(a3);
\draw (a6)--(b3);
\draw (a4)--(a5);
\draw (a5)--(a6);
\draw (a1)--(b1);
\draw (a3)--(b3);
\end{tikzpicture}
}
\caption{Graph $C_6$ and Graph $AmB_A(C_6 \mid J)$}\label{fig:AmB_C6}
\end{figure}

A graph is said to be \textbf{planar} if it can be drawn in the plane in such a way that no two edges intersect except at their common endpoints. Such a representation is called a plane drawing of the graph. In other words, a graph is planar if it admits an embedding in the plane where edges meet only at vertices to which they are incident. The examples of planar graphs include trees, cycle graphs, and wheel graphs, although not every graph is planar. A deeper understanding of graph-theoretic concepts and their properties can be found in standard references such as \cite{diestel, wilson}.

In this paper, we study the generalized von Neumann inverse graph over finite commutative von Neumann regular rings, with particular attention to its reduced form $\Gamma'_{Reg}(R)$. The discussion focuses on revealing how the structure of the ring influences the graph, especially in terms of connectivity, acyclicity, and planarity. We also examine several structural properties, including vertex degrees, girth, and the presence of pendant vertices, and relate these features to the underlying algebraic characteristics of $R$. In addition, we identify conditions under which the graph takes the form of familiar classes such as paths, cycles, and wheels, and determine when it contains an induced subgraph that is isomorphic to the house graph. The main contribution of this paper is the development of an explicit algorithm that determines the structure of $\Gamma'_{Reg}(R)$ for finite commutative von Neumann regular rings in a general setting. Furthermore, we establish a structural relationship between $\Gamma'_{Reg}(R)$ and the inclusion ideal graph of $R$, which provides a deeper understanding of how ideal-theoretic properties are reflected in the graph. This algorithmic approach not only simplifies the analysis but also serves as a unifying framework for describing the graph in terms of the underlying ring structure.

\section{Result and Discussion}
\begin{definition}
    Let $R$ be a von Neumann regular ring. The generalized von Neumann inverse graph of $R$, denoted by $\Gamma_{Reg}(R)$, where the vertex set is $V(\Gamma_{Reg}(R))=Reg(R)=R$, and two distinct vertices $a,b \in R$ are adjacent if and only if $aba=a$ or $bab=b$.
\end{definition}
Let $a$ and $b$ be elements of a von Neumann regular ring. The element $a$ is said to be regularized by $b$ if $aba=a$, or $b$ is the inner inverse of $a$. Moreover, $a$ is said to regularize $b$ if $bab=b$. In this context, given a von Neumann regular ring $R$, the graph $\Gamma_{Reg}(R)$ represents the relationships among the elements of $R$ that regularize another element or are regularized by one another. In this discussion, we restrict our attention to finite von Neumann regular rings.


We construct the induced subgraph $\Gamma'_{Reg}(R) \leq \Gamma_{Reg}(R)$ on the vertex set $R \setminus \{0_R\}$. Consequently, the structure of $\Gamma_{Reg}(R)$ is obtained as described in the following theorem.

\begin{theorem}\label{teo_strukgrafvn}
    Let $R$ be a ring. Then $\Gamma_{Reg}(R) \cong K_1 + \Gamma'_{Reg}(R)$. 
\end{theorem}
\begin{proof}
    Let $R$ be a ring. It is known that $V(\Gamma_{Reg}(R))=R=\{0_R\} \cup (R \setminus \{0_R\})=\{0_R\} \cup V(\Gamma'_{Reg}(R))$. Since $\Gamma'_{Reg}(R)$ is an induced subgraph of $\Gamma_{Reg}(R)$ and for every $a \in R$, we have $0_Ra0_R=0_R$, that is, $0_R$ is regularized by $a$. It follows that $E(\Gamma_{Reg}(R))=E(\langle \{0_R\}\rangle_{\Gamma_{Reg}(R)}) \cup E(\Gamma'_{Reg}(R)) \cup \{0_Ra: a \in R \setminus \{0_R\}\}=E(\langle \{0_R\}\rangle_{\Gamma_{Reg}(R)}) \cup E(\Gamma'_{Reg}(R)) \cup \{0_Ra: a \in V(\Gamma'_{Reg}(R)\}$. Hence, $\Gamma_{Reg}(R) = \langle \{0_R\}\rangle_{\Gamma_{Reg}(R)} + \Gamma'_{Reg}(R) \cong K_1 + \Gamma'_{Reg}(R)$.
\end{proof}

Before proceeding further, it is important to note that the construction of the generalized von Neumann inverse graph depends only on the algebraic structure of the underlying ring up to isomorphism. In other words, rings that are structurally identical give rise to graphs with the same structure. This observation is formalized in the following theorem.

\begin{theorem}
    Let $R_1$ and $R_2$ be rings with $R_1 \cong R_2$. Then it follows that $\Gamma_{Reg}(R_1) \cong \Gamma_{Reg}(R_2)$ and $\Gamma'_{Reg}(R_1) \cong \Gamma'_{Reg}(R_2)$.
\end{theorem}
\begin{proof}
    Since $R_1 \cong R_2$, there exists a ring isomorphism $f:R_1 \to R_2$. Moreover, $f$ induced a corresponding map  $f:V(\Gamma_{Reg}(R_1)) \to V(\Gamma_{Reg}(R_2))$. For arbitrary elements $a_1,b_1 \in V(\Gamma_{Reg}(R_1))$. Then $a_1b_1 \in E(\Gamma_{Reg}(R_1))$ if and only if $f(a_1)f(b_1) \in E(\Gamma_{Reg}(R_2))$.
    Hence, $\Gamma_{Reg}(R_1) \cong \Gamma_{Reg}(R_2)$.
    Furthermore, by the property of ring isomorphism, we have $f(0_{R_1})=0_{R_2}$. Consequently, we can define the restriction $f'|_{R_1 \setminus \{0_{R_1}\}}$. In other words, $f':V(\Gamma'_{Reg}(R_1)) \to V(\Gamma_{Reg}(R_2))$ where $f'(x)=f(x)$ for every $x \in V(\Gamma'_{Reg}(R_1)) = R_1 \setminus \{0_{R_1}\}$. For arbitrary elements $a_1,b_1 \in V(\Gamma'_{Reg}(R_1))=R_1 \setminus \{0_{R_1}\}$. Then, $a_1b_1 \in E(\Gamma'_{Reg}(R_1))$ if and only if $f'(a_1)f'(b_1) \in E(\Gamma'_{Reg}(R_2))$.
    Hence, $\Gamma'_{Reg}(R_1) \cong \Gamma'_{Reg}(R_2)$. 
\end{proof}

Based on Theorem \ref{teo_strukgrafvn}, the structure of the graph $\Gamma_{Reg}(R)$ for a ring $R$ depends solely on the graph $\Gamma'_{Reg}(R)$. Therefore, in the following discussion, we will examine the properties and characteristics of the graph $\Gamma'_{Reg}(R)$. 

The structure of finite commutative von Neumann regular rings can be understood through a sequence of fundamental results linking regularity, nilpotent elements, and ideal structure. We begin with the observation that the class of von Neumann regular rings is stable under finite direct products. Indeed, if $F_1, F_2, \dots, F_k$ are fields, then each element in the product $F_1 \times F_2 \times \dots \times F_k$ admits a coordinate-wise 'quasi-inverse', showing that the product itself is von Neumann regular. This provides a basic source of examples and suggests that such products play a central role in the general theory.

To analyze the internal structure of a commutative von Neumann regular ring $R$, one considers its nilradical $Nil(R)$, defined as the set of all nilpotent elements. A key property in this setting is that $R$ contains no nonzero nilpotent elements, that is, $Nil(R)={0}$. The proof relies on the defining condition of regularity: for any $a \in R$, there exists $x \in R$ such that $a=a^2x$. If $a$ is nilpotent, say $a^n=0$, repeated substitution into this relation forces $a$ itself to vanish. This shows that commutative von Neumann regular rings are reduced.

The behavior of prime ideals in such rings is particularly restrictive. If $P$ is a prime ideal of a commutative von Neumann regular ring $R$, then the quotient $R/P$ is not only an integral domain but in fact a field. This follows from the regularity condition, which guarantees that every nonzero element in $R/P$ has a multiplicative inverse. Consequently, every prime ideal is maximal. This observation is crucial, as it allows one to treat the set of prime ideals as a finite collection of maximal ideals in the finite case.

Another important ingredient is the general fact that, for any ring $R$, the nilradical coincides with the intersection of all prime ideals of $R$. When combined with the previous result that $Nil(R)={0}$, it follows that the intersection of all prime ideals in a commutative von Neumann regular ring is trivial. In the finite case, where there are only finitely many prime (hence maximal) ideals $P_1,P_2,\dots,P_n$, we obtain
$P_1 \cap P_2 \cap \dots \cap P_n = {0}$.

Since distinct maximal ideals are comaximal, the Chinese Remainder Theorem applies. It yields an isomorphism
$R \cong R/(P_1 \cap P_2 \cap \dots \cap P_n) \cong R/P_1 \times R/P_2 \times \dots \times R/P_n$.
Each quotient $R/P_i$ is a field, and therefore $R$ decomposes as a finite direct product of fields. This establishes that every finite commutative von Neumann regular ring admits such a decomposition.

Conversely, any finite direct product of fields is itself a commutative von Neumann regular ring, as noted at the outset. Combining these two directions, we obtain a complete characterization: a finite commutative ring is von Neumann regular if and only if it is isomorphic to a finite direct product of fields. This result highlights the rigid algebraic structure of such rings and provides a concrete framework for further analysis. 

In the subsequent discussion, we restrict our attention to finite commutative von Neumann regular rings. As a starting point, it is natural to examine the graph structure in the simplest case, namely when the ring is a field. This case serves as a fundamental building block for understanding more general structures. The following theorem describes the form of $\Gamma'_{Reg}(F)$ for an arbitrary field $F$.

\begin{theorem}\label{teo_grafvn_lap}
    For every field $F$, then $\Gamma'_{Reg}(F) \cong |U'(F)|K_1 \cup \frac{|U''(F)|}{2} K_2$.
\end{theorem}
\begin{proof}
    Let $a,b \in V(\Gamma'_{Reg}(F))=F \setminus \{0_F\}$ be arbitrary elements with $ab \in E(\Gamma'_{Reg}(F))$. This means that $aba=a$ or $bab=b$. It follows that $a=b^{-1}$. Consequently, every nonzero element of $F$ that is self-inverse forms an isolated vertex. We have that $U'(F)$ is the set of unit elements in $F$ which are self-inverse, and $U''(F)$ is the set of unit elements in $F$ whose inverses are different from themselves. Since the inverse of a unit element is unique, it follows that $\langle U'(F) \rangle _{\Gamma'_{Reg}(F)}=|U'(F)|K_1$ and $\langle U''(F) \rangle _{\Gamma'_{Reg}(F)}=\frac{|U''(F)|}{2}K_2$. Hence $\Gamma'_{Reg}(F)= \langle U'(F) \rangle _{\Gamma'_{Reg}(F)} \cup \langle U''(F) \rangle _{\Gamma'_{Reg}(F)}= |U'(F)|K_1 \cup \frac{|U''(F)|}{2} K_2$.
\end{proof}


Having established the structure of $\Gamma'_{Reg}(F)$ for a field, we now extend the discussion to arbitrary finite commutative von Neumann regular rings. By the decomposition $R \cong F_1 \times F_2 \times \dots \times F_n$, where each $F_i$ is a field, the structure of $\Gamma'_{Reg}(R)$ can be analyzed in terms of its components. In particular, this decomposition enables us to determine fundamental graph parameters, such as the degree of vertices, in a systematic way. The next theorem provides a characterization of the vertex degrees of $\Gamma'_{Reg}(R)$ under this representation.

\begin{theorem}\label{teo_deggrafvn'}
    Let $R$ be a finite commutative von Neumann regular ring. Suppose that $R \cong F_1 \times F_2 \times \dots \times F_n$, where $F_1, F_2, \dots, F_n$ are fields. The degree of $(a_1,a_2,\dots,a_n)$, where $a_{i_1}=0_{F_{i_1}}, a_{i_2}=0_{F_{i_2}}, \dots, a_{i_k}=0_{F_{i_k}}$ for $k > 0$ and $1\leq i_1,i_2,\dots,i_k \leq n$, is given by $$\deg_{\Gamma'_{Reg}(R)}((a_1,a_2,\dots,a_n))=\begin{cases}
        2^{n-k}-3&+|F_{i_1}||F_{i_2}|\dots|F_{i_k}|, \\
        &\text{if } a_i^{-1}=a_i \text{ for every } a_i \neq 0_{F_i}\\
        2^{n-k}-2&+|F_{i_1}||F_{i_2}|\dots|F_{i_k}|, \\&\text{otherwise}.
    \end{cases}$$
    Furthermore, if $k=0$, the degree of the vertex $(a_1,a_2,\dots,a_n)$ in the graph $\Gamma'_{Reg}(R)$ is $$\deg_{\Gamma'_{Reg}(R)}((a_1,a_2,\dots,a_n))=\begin{cases}
        2^{n}-2, &\text{if } a_i^{-1}=a_i \text{ for every } a_i \neq 0_{F_i}\\
        2^{n}-1, &\text{otherwise}.
    \end{cases}$$
\end{theorem}
\begin{proof}
    For an arbitrary vertex $(a_1,a_2,\dots,a_n) \in F_1 \times F_2 \times \dots \times F_n$, with $a_{i_1}=0_{F_{i_1}}, a_{i_2}=0_{F_{i_2}}, \dots ,a_{i_k}=0_{F_{i_k}}$ for some $k > 0$ and $1\leq i_1,i_2,\dots,i_k \leq n$. The element $(a_1,a_2,\dots,a_n)$ is regularized by $(b_1,b_2,\dots,b_n)$ where $b_i=a_i^{-1}$ for every $i \in \{1,2,\dots,n\} \setminus \{i_1,i_2,\dots,i_k\}$ and $b_{j} \in F_j$ for every $j \in \{i_1,i_2,\dots,i_k\}$. In addition, the element $(a_1,a_2,\dots,a_n)$ regularizes the elements in the set $$\prod_{i=1}^n B_i \setminus \{(0_{F_1},0_{F_2} \dots,0_{F_n})\} \text{ with }B_i=\begin{cases}
        \{a_i^{-1}, 0_{F_i}\}, &\text{if } i \notin \{i_1,i_2,\dots,i_k\}\\
        \{0_{F_i}\}, &\text{if }i \in \{i_1, i_2, \dots, i_k\}.
    \end{cases}$$
Since $(a_1,a_2,\dots,a_n)$ only regularizes and is regularized at the same time by $(b_1,b_2,\dots,b_n)$, the number of elements that are regularized by and regularize the element $(a_1,a_2,\dots,a_n)$ is $|F_{i_1}||F_{i_2}|\dots|F_{i_k}| + (2^{n-k}-1)-1= 2^{n-k}-2 +|F_{i_1}||F_{i_2}|\dots|F_{i_k}|$. If $a_i^{-1}=a_i$ for every $i=1,2,\dots,n$, then the number of distinct elements that are regularized by and regularize the element $(a_1,a_2,\dots,a_n)$ is $2^{n-k}-2 +|F_{i_1}||F_{i_2}|\dots|F_{i_k}|-1=2^{n-k}-3 +|F_{i_1}||F_{i_2}|\dots|F_{i_k}|$. Next, take an arbitrary point $(a_1,a_2,\dots,a_n) \in F_1 \times F_2 \times \dots F_n$. If $a_i \neq 0_{F_i}$ for every $i=1,2,\dots,n$, then $a_i$ is only regularized by $a_i^{-1}$ and $a_i$ regularizes $0_{F_i}$ in $F_i$. Consequently, $(a_1,a_2,\dots,a_n)$ is only regularized by $(a_1^{-1},a_2^{-1},\dots, a_n^{-1})$. On the other hand, $(a_1,a_2,\dots,a_n)$ regularizes the elements in the set $\prod_{i=1}^n\{a_i^{-1},0_{F_i}\} \setminus \{(0_{F_1},0_{F_2} \dots,0_{F_n})\}$. Hence, the number of elements that are regularized by and regularize the element $(a_1,a_2,\dots,a_n)$ is $2^n-1$. If $a_i^{-1}=a_i$ for every $i=1,2,\dots,n$, then $(a_1,a_2,\dots,a_n)=(a_1^{-1},a_2^{-1},\dots, a_n^{-1})$, so the number of distinct elements that are regularized by and regularize the element $(a_1,a_2,\dots,a_n)$ is $2^n-2$. 
\end{proof}

With the characterization of vertex degrees in hand, we proceed to examine the global structure of $\Gamma'_{Reg}(R)$, in particular, its connectivity. Understanding when this graph is connected provides further insight into how the decomposition of $R$ influences its overall structure. The following theorem gives a necessary and sufficient condition for $\Gamma'_{Reg}(R)$ to be a nontrivial connected graph.

\begin{theorem}\label{teo_grafvnterhubungnontr}
    Let $R \cong F_1 \times F_2 \times \dots \times F_n$ be a ring for some $n \geq 1$, where $F_1, F_2, \dots, F_n$ are fields. The graph $\Gamma'_{Reg}(R)$ is a nontrivial connected graph if and only if $n \geq 2$.
\end{theorem}
\begin{proof}
    Let $R \cong F_1 \times F_2 \times \dots \times F_n$ be a ring for some $n \geq 1$, where $F_1, F_2, \dots, F_n$ are fields.
    \begin{enumerate}
        \item[$(\Longleftarrow)$] Let $n \geq 2$ be given. Take arbitrary $(a_1,a_2,\dots,a_n), (b_1,b_2,\dots,b_n) \in V(\Gamma'_{Reg}(R))= (F_1 \times F_2 \times \dots \times F_n) \setminus \{(0_{F_1},0_{F_2},\dots,0_{F_n})\}$. This means that there exist $a_i \neq 0_{F_i}$ and $b_j \neq 0_{F_j}$ for some $1 \leq i,j \leq n$. 
        Suppose that
        \begin{align*}
            \textbf{e}^k&=(e^k_1,e^k_2,\dots,e^k_n) \text{ for every } 1 \leq k \leq n, \text{ where }\\& \qquad e^k_c=\begin{cases}
            1_{F_c}, &\text{if } c \neq k\\
            0_{F_k}, &\text{if } c = k
        \end{cases} \text{ for } 1 \leq c \leq n,\\
        \textbf{x}^i&=(x_1,x_2,\dots,x_n), \text{ where } x_c=\begin{cases}
            0_{F_c}, &\text{if } c \neq i\\
            a_i^{-1}, &\text{if }c =i
        \end{cases} \text{ for } 1 \leq c \leq n,\\
        \textbf{y}^i&=(y_1,y_2,\dots,y_n), \text{ where } y_c=\begin{cases}
            1_{F_c}, &\text{if } c \neq i\\
            a_i, &\text{if }c=i
        \end{cases} \text{ for } 1 \leq c \leq n,\\
        \textbf{u}^j&=(u_1,u_2,\dots,u_n), \text{ where } u_c=\begin{cases}
            0_{F_c}, &\text{if } c \neq j\\
            b_j^{-1}, &\text{if }c =j
        \end{cases} \text{ for } 1 \leq c \leq n,\\
        \textbf{v}^j&=(v_1,v_2,\dots,v_n), \text{ where } v_c=\begin{cases}
            1_{F_c}, &\text{if } c \neq j\\
            b_j, &\text{if }c=j
        \end{cases} \text{ for } 1 \leq c \leq n.
        \end{align*}
        It is observed that a path can be constructed as follows 
        \begin{align*}
            (a_1,a_2,\dots,a_n) - \textbf{x}^i- \textbf{y}^i - \textbf{e}^{i} - (1,1, \dots, 1) - \textbf{e}^{j} -  \textbf{v}^j - \textbf{u}^j - (b_1,b_2,\dots,b_n).
        \end{align*}
        Hence, the graph $\Gamma'_{Reg}(R)$ is connected. Since $|V(\Gamma'_{Reg}(R))| \geq 2^n-1 \geq 3$, the graph $\Gamma'_{Reg}(R)$ is nontrivial.
        \item[$(\Longrightarrow)$] if $n=1$, then $R \cong F_1$. Based on Theorem \ref{teo_grafvn_lap}, we obtain the graph $\Gamma'_{Reg}(R)=|U'(F_1)|K_1 \cup \frac{|U''(F)|}{2}K_2$. If $|R \setminus \{0\}|>1$, then the graph $\Gamma'_{Reg}(R)$ is disconnected. If $|R \setminus \{0\}|=1$, then the graph $\Gamma'_{Reg}(R)$ is the trivial graph $K_1$.
    \end{enumerate}
\end{proof}

Having determined the conditions under which $\Gamma'_{Reg}(R)$ is connected, it is natural to investigate more refined structural features of the graph, particularly those related to its local configuration. One such feature is the existence of pendant vertices, which often reflects specific constraints on the underlying ring. The following theorem characterizes precisely when $\Gamma'_{Reg}(R)$ contains a pendant vertex.

\begin{theorem}\label{teo_grafvn_titikpendant}
    Let $R$ be a ring. The graph $\Gamma'_{Reg}(R)$ has a pendant vertex if and only if $R \cong F_1$ or $R\cong \mathbb{F}_2 \times F$, where $F_1$ and $F$ are fields with $|F_1| \geq 4$.
\end{theorem}
\begin{proof}
    Given an arbitrary ring $R$, and fields $F, F_1$ with $|F_1|>4$.
\begin{enumerate}
    \item[$(\Longleftarrow)$] If $R \cong F_1$, then $\Gamma'_{Reg}(R)=|U'(R)|K_1 \cup \frac{|U''(R)|}{2}K_2$. Since $|R| \geq 4$ and $|U'(R)|\leq2$, we obtain $|U''(R)|>0$. The vertices in the subgraph $\frac{|U''(R)|}{2}K_2$ are pendant vertices. if $R \cong \mathbb{F}_2 \times F$, then there exists a vertex $(0,1_F) \in \mathbb{F}_2 \times F$ such that, by Theorem \ref{teo_deggrafvn'}, we obtain $\deg_{\Gamma'_{Reg}(R)}((0,1_F))=2^1-3+2=1$. Hence, the vertex $(0,1_F) \in V(\Gamma'_{Reg}(R))$ is a pendant vertex.
    \item[$(\Longrightarrow)$] Suppose $R \cong F_1 \times F_2 \times \dots \times F_n$, for some $n \geq 1$, where $F_1, F_2, \dots, F_n$ are fields. Suppose there exists a pendant vertex $a \in V(\Gamma'_{Reg}(R))$. By Theorem \ref{teo_deggrafvn'}, $1=\deg_{\Gamma'_{Reg}(R)}(a) \in \{2^{n}-2, 2^{n}-1, 2^{n-k}-3+|F_{i_1}||F_{i_2}|\dots|F_{i_k}|, 2^{n-k}-2+|F_{i_1}||F_{i_2}|\dots|F_{i_k}|: 0<k < n \}$. Consider the following possibilities.
    \begin{enumerate}
        \item If $2^n-2=1$, then $2^n=3$. This is a contradiction.
        \item If $2^n-1=1$, then $2^n=2$ so that $n=1$. Hence, $R=F_1$.
        \item If $2^{n-k}-3+|F_{i_1}||F_{i_2}|\dots|F_{i_k}|=1$, then $ 2^{n-k}+|F_{i_1}||F_{i_2}|\dots|F_{i_k}|=4$. Since $|F_{i_1}||F_{i_2}|\dots|F_{i_k}| \geq 2^k$, we obtain $4-2^{n-k} \geq 2^k \iff 2^k+2^{n-k} \leq 4$. Consequently, we must have $1=k=n-1$, so that $k=1$ and $n=2$. Hence, $2+|F_{i_1}|=4 \iff |F_{i_1}|=2$. Thus, $R \cong F_{i_1} \times F_2 \cong \mathbb{F}_2 \times F_2$, since every field of cardinality two is always isomorphic to the field $\mathbb{F}_2$. 
        \end{enumerate}
    \end{enumerate}
\end{proof}

After identifying when $\Gamma'_{Reg}(R)$ contains a pendant vertex, we examine this situation in more detail. It is natural to ask how many pendant vertices may occur when the graph is connected. The following proposition provides a precise description of the number of pendant vertices of $\Gamma'_{Reg}(R)$ in terms of the structure of the ring.

\begin{proposition}\label{prop_byktitikdaunvn}
    Let $R$ be a ring such that the graph $\Gamma'_{Reg}(R)$ is connected and has a pendant vertex. Then the following hold
    \begin{enumerate}
        \item The number of the pendant vertices of the graph $\Gamma'_{Reg}(R)$ is $1$ if and only if $R \cong \mathbb{F}_2 \times \mathbb{F}_{2^k}$ for some $k \in \mathbb{N}\setminus\{1\}$.
        \item The number of the pendant vertices of the graph $\Gamma'_{Reg}(R)$ is $2$ if and only if $R \cong \mathbb{F}_2 \times \mathbb{F}_2$ or $\mathbb{F}_2 \times \mathbb{F}_{p^k}$, where $k \in \mathbb{N}$ and $p$ is an odd prime number.
    \end{enumerate}
\end{proposition}
\begin{proof}
    Let $R$ be a ring such that graph $\Gamma'_{Reg}(R)$ is connected and has a pendant vertex. By Theorem \ref{teo_grafvnterhubungnontr} and Theorem \ref{teo_grafvn_titikpendant}, it follows that $R \cong \mathbb{F}_2 \times F$, where $F$ is a field. We obtain $V(\Gamma'_{Reg}(R))=\{(1,0_F),(0,a),(1,a): a \in F \setminus \{0_F\}\}$, with
    \begin{align*}
        \deg_{\Gamma'_{Reg}(R)}(1,0_F)&=2-3+|F|=|F|-1,\\ \deg_{\Gamma'_{Reg}(R)}(0,a)&=\begin{cases}
            2-3+2, &\text{if } a^{-1}=a\\
            2-2+2, &\text{if } a^{-1}\neq a 
        \end{cases}=\begin{cases}
            1, &\text{if } a^{-1}=a\\
            2, &\text{if } a^{-1}\neq a,
        \end{cases}\\  
        \deg_{\Gamma'_{Reg}(R)}(1,a)&=\begin{cases}
            2^2-2, &\text{if } a^{-1}=a\\
            2^2-1, &\text{if } a^{-1}\neq a 
        \end{cases}=\begin{cases}
            2, &\text{if } a^{-1}=a\\
            3, &\text{if } a^{-1}\neq a,
        \end{cases}
    \end{align*}
    for every $a \in F \setminus \{0_F\}$. 
    Hence, the vertices that can serve as pendant vertices are $(1,0_F)$ with $|F|=2$ or $(0,a)$ with $a \in U'(F)$. If $F \cong \mathbb{F}_2$, then there are exactly two pendant vertices in the graph $\Gamma'_{Reg}(R)$, namely $(0,1)$ and $(1,0)$. If $F \cong \mathbb{F}_{p^k}$ for some $k \in \mathbb{N}$ where $p$ is an odd prime, there are exactly two pendant vertices in the graph $\Gamma'_{Reg}(R)$, namely $(0,1_F)$ and $(0,-1_F)$. If $F \ncong \mathbb{F}_2$ and $F \ncong \mathbb{F}_{p^k}$ for every $k \in \mathbb{N}$ and odd prime $p$, then $F \cong \mathbb{F}_{2^k}$ for some $k \in \mathbb{N}\setminus\{1\}$. Consequently, $|F|>2$, and there is exactly one pendant vertex in the graph $\Gamma'_{Reg}(R)$, namely $(0,1_F)$.
\end{proof}

The previous proposition determines how many pendant vertices may appear under the given conditions. The next result focuses on the behavior of such a vertex within the graph. In particular, it describes how a pendant vertex interacts with other vertices through the regularization relation in $\Gamma'_{Reg}(R)$.

\begin{theorem}\label{teo_pendantvert}
    Let $R$ be a ring. If the graph $\Gamma'_{Reg}(R)$ is connected and has a pendant vertex $a$, then $a$ is regularized only by itself and by one element that is adjacent to $a$. Furthermore, $a$ does not regularize any other element in the graph $\Gamma'_{Reg}(R)$.
\end{theorem}
\begin{proof}
    Let $R$ be a ring such that the graph $\Gamma'_{Reg}(R)$ is connected and has a pendant vertex. By Theorem \ref{teo_grafvnterhubungnontr} and Theorem \ref{teo_grafvn_titikpendant}, we have $R \cong \mathbb{F}_2 \times F$, where $F$ is a field. Suppose that $a \in V(\Gamma'_{Reg}(R))$ is a pendant vertex. Consider the following cases.
    \begin{enumerate}
        \item If $F \cong \mathbb{F}_2$, then $V(\Gamma'_{Reg}(R))=\{(0,1),(1,0),(1,1)\}$. Using Proposition \ref{prop_byktitikdaunvn}, we obtain $a \in \{(0,1),(1,0)\}$. The elements $(0,1)$ and $(1,0)$ respectively regularize and are regularized by the elements $(0,1)$ and $(1,0)$. Furthermore, the elements $(0,1)$ and $(1,0)$ are regularized by the element $(1,1)$, but the elements $(0,1)$ and $(1,0)$ do not regularize the element $(1,1)$.
        \item If $F \cong \mathbb{F}_{p^k}$, for some $k \in \mathbb{N}$ and an odd prime number $p$, then by Proposition \ref{prop_byktitikdaunvn}, we have $a \in \{(0,1_F),(0,-1_F)\}$. It is known that $(0,1_F)(1,1_F), (0,-1_F)(1,-1_F) \in E(\Gamma'_{Reg}(R))$, since $(0,1_F)(1,1_F)(0,1_F)=(0,1_F)$ and $(0,-1_F)(1,-1_F)(0,-1_F)=(0,-1_F)$. However, $$(1,1_F)(0,1_F)(1,1_F) \neq (1,1_F) \text{ and } (1,-1_F)(0,-1_F)(1,-1_F) \neq (1,-1_F).$$ 
        On the other hand, the elements $(0,1_F)$ and $(0,-1_F)$ regularize and are regularized by themselves. Suppose that there exists another element that is regularized by $a$, then this leads to a contradiction with the fact that $a$ is a pendant vertex.
        \item If $F \cong \mathbb{F}_{2^k}$, then by Proposition \ref{prop_byktitikdaunvn}, we have $a = (0,1_F)$. It is known that $(0,1_F)(1,1_F)\in E(\Gamma'_{Reg}(R))$, since $(0,1_F)(1,1_F)(0,1_F)=(0,1_F)$. However, $(1,1_F)(0,1_F)(1,1_F) \neq (1,1_F)$. On the other hand, the element $(0,1_F)$ regularizes and is regularized by itself. Suppose that there exists another element that is regularized by $a$, then this leads to a contradiction with the fact that $a$ is a pendant vertex.
    \end{enumerate}
    Therefore, $a$ is regularized by itself and by one element that is adjacent to $a$, but $a$ does not regularize any other element in the graph $\Gamma'_{Reg}(R)$.
\end{proof}

The structural properties established in the previous theorem lead directly to further restrictions on the ring and its elements. In particular, the presence of a pendant vertex imposes a specific form on $R$ and reveals additional relationships among the pendant vertices themselves. The following corollary summarizes these consequences.

\begin{corollary}
    Let $R$ be a ring such that the graph $\Gamma'_{Reg}(R)$ is connected and has a pendant vertex. Then $R \cong \mathbb{F}_2 \times F$, where $F$ is a field. If $A$ denotes the set of all pendant vertices, then for every $a,b \in A$, we have $[ a ] = [ b ]$ or $a$ and $b$ are idempotent elements.
\end{corollary}
\begin{proof}
    Based on the proof of Theorem \ref{teo_pendantvert}, it is obtained that 
    \begin{enumerate}
        \item if $F \cong \mathbb{F}_2$, then $A=\{(0,1),(1,0)\}$, and hence $A \subseteq Id(R)$,
        \item if $F \cong \mathbb{F}_{p^k}$ with $k \in \mathbb{N}$ and $p$ an odd prime number, then $A=\{(0,1_{F}),(0,-1_{F})\}$. It follows that $(0,1_{F})R=(0,-1_{F})R$ since $(0,-1_F)^3=(0,-1_F)$ and $(0,-1_F)^2=(0,1_F)$,
        \item if $F \cong \mathbb{F}_{2^k}$ with $k \in \mathbb{N}$, then $A=\{(0,1_F)\}$, and hence $A \subseteq Id(R)$.
    \end{enumerate}
\end{proof}


The previous corollary describes the structure of rings whose graphs contain pendant vertices. Building on these observations, we now turn to a more specific class of graphs, namely trees. The following theorem gives a complete characterization of rings for which $\Gamma'_{Reg}(R)$ forms a nontrivial tree.

\begin{theorem}
    Let $R$ be a ring. The graph $\Gamma'_{Reg}(R)$ is a nontrivial tree if and only if $R \cong {\mathbb{F}_2 \times \mathbb{F}_2}$ or $R \cong \mathbb{F}_2 \times \mathbb{F}_3$.
\end{theorem}
\begin{proof}
    Let $R$ be a ring.
    \begin{enumerate}
        \item[$(\Longleftarrow)$] If $R \cong \mathbb{F}_2 \times 
        \mathbb{F}_2$, then $\Gamma'_{Reg}(R) \cong P_3$. If $R \cong \mathbb{F}_2 \times \mathbb{F}_3$, then $\Gamma'_{Reg}(R) \cong P_5$. Thus, the graphs $\Gamma'_{Reg}(\mathbb{F}_2 \times 
        \mathbb{F}_2)$ and $\Gamma'_{Reg}(\mathbb{F}_2 \times 
        \mathbb{F}_3)$ are trees.
        \item[$(\Longrightarrow)$] If the graph $\Gamma'_{Reg}(R)$ is a nontrivial tree, then $\Gamma'_{Reg}(R)$ is a nontrivial connected graph. By Theorem \ref{teo_grafvnterhubungnontr}, we obtain $R \cong F_1 \times F_2 \times \dots, \times F_n$, where $F_1, F_2, \dots, F_n$ are fields and $n \geq 2$. Assume that $R \ncong \mathbb{F}_2 \times \mathbb{F}_2$ and $R \ncong \mathbb{F}_2 \times \mathbb{F}_3$. Then the following possibilities arise.
        \begin{enumerate}
            \item Case $R \cong \mathbb{F}_2 \times F_2$, where $F_2$ is a field and $|F_2| > 3$. This means that there exists $a \in F_2$ such that $a^{-1} \neq a$. Consequently, the vertices $(1,0_{F_2})$, $(1,a)$, and $(1,a^{-1})$ in the graph $\Gamma'_{Reg}(R)$ form a cycle. This contradicts the fact that $\Gamma'_{Reg}(R)$ is a tree.
            \item Case $R \cong \mathbb{F}_3 \times \mathbb{F}_3$. There exist $1,-1 \in U'(\mathbb{F}_3)$, so a cycle $(1,0) - (1,1)-(0,1)-(-1,1)-(-1,0)-(-1,-1)-(0,-1)-(1,-1)-(1,0)$ is formed in the graph $\Gamma'_{Reg}(R)$. This contradicts the fact that $\Gamma'_{Reg}(R)$ is a tree.
            \item Case $R \cong F_1 \times F_2$, where $F_1$ and $F_2$ are fields and $\text{maks}\{|F_1|,|F_2|\}>3$. Without loss of generality assume that $|F_1|>3$, which means that there exists $a \in F_1$ such that $a^{-1} \neq a$. Hence there is a cycle $(a,1_{F_2})-(a^{-1},1_{F_2})-(0,1_{F_2})$. This contradicts the fact that $\Gamma'_{Reg}(R)$ is a tree.
            \item Case $R \cong F_1 \times F_2 \times \dots \times F_n$, where $F_1,F_2,\dots,F_n$ are fields and $n \geq 3$. The vertices $(1_{F_1},1_{F_2},\dots,1_{F_n}), (1_{F_1},0_{F_2},\dots,0_{F_n})$, and $(1_{F_1},1_{F_2},0_{F_3},\dots,0_{F_n})$ in the graph $\Gamma'_{Reg}(R)$ form a cycle. This contradicts the fact that $\Gamma'_{Reg}(R)$ is a tree.
        \end{enumerate}
        Therefore, it must be that $R \cong \mathbb{F}_2 \times \mathbb{F}_2$ or $R \cong \mathbb{F}_2 \times \mathbb{F}_3$.
    \end{enumerate}
\end{proof}

The characterization of $\Gamma'_{Reg}(R)$ as a tree in the previous theorem provides a useful starting point to examine its cycle structure. In particular, it allows us to determine the girth of the graph in the nontrivial connected case. The following corollary gives a complete description of the girth of $\Gamma'_{Reg}(R)$ based on the structure of the ring.

\begin{corollary}
    Let $R$ be a ring such that the graph $\Gamma'_{Reg}(R)$ is a nontrivial connected graph. Then
    $$\text{girth}(\Gamma'_{Reg}(R))=\begin{cases}
        \infty, &\text{if } R \cong \mathbb{F}_2 \times \mathbb{F}_2 \text{ or } R \cong \mathbb{F}_2 \times \mathbb{F}_3\\
        8, &\text{if } R \cong \mathbb{F}_3 \times \mathbb{F}_3\\
        3, &\text{otherwise}.
    \end{cases}$$
\end{corollary}

The preceding result highlights how the presence of cycles influences the overall structure of $\Gamma'_{Reg}(R)$. To further refine this analysis, it is useful to consider specific induced subgraphs that impose stronger structural restrictions. One such example is the house graph. The following lemma shows that the existence of this induced subgraph prevents the graph from belonging to several well-known graph classes.

\begin{lemma}\label{lemmasubgrafrumahCnWm}
    If a graph $G$ contains an induced subgraph that is isomorphic to a house graph, then $G \ncong P_k$, $G \ncong C_n$, and $G \ncong W_m$ for any natural numbers $k,n,m$.
\end{lemma}
\begin{proof}
    If a graph $G$ contains an induced subgraph that is isomorphic to a house graph, say $A$, then there exists a vertex $x \in V(A)$ with $\deg_{A}(x)=3$, and hence $\deg_{A}(x)\geq 3$. Since every vertex degree in a path graph is at most $2$ and every vertex degree in a cycle graph is $2$, it follows that $G \ncong P_k$ and $G \ncong C_n$ for any natural numbers $k,n$. Next, suppose that $G \cong W_m$ for some natural number $m$. This means that there exists $a \in V(G)$ such that $ab \in E(G)$ for every $b \in V(G)\setminus\{a\}$. If $a \in V(A)$, then it must be that $\deg_{A}(a)=4$ since $|V(A)|=5$. This contradicts a property of the house graph, namely that $\deg_A(v)\leq 3$ for every $v \in V(A)$. Observe that for every $y \in V(A)$ we have $\deg_G(y)=3$. If $a \notin V(A)$, then $\deg_{A}(y) \leq 2$. This contradicts the fact that there exists a vertex in the house graph $A$ with degree $3$. Hence, $G \ncong W_m$ for any natural number $m$.
\end{proof}

The previous lemma shows that the presence of a house graph as an induced subgraph rules out several familiar graph structures. It is therefore important to identify conditions under which such a subgraph appears in $\Gamma'_{Reg}(R)$. The following lemma provides a criterion based on the sizes of the component fields in the decomposition of $R$.

\begin{lemma}\label{lemmamaksFgrafrumah}
    Let $R \cong F_1 \times F_2 \times \dots \times F_n$ be a ring, where $F_1, F_2, \dots, F_n$ are fields and $n \geq 2$. If $\text{maks}\{|F_1|,|F_2|,\dots,|F_n|\}> 3$, then the graph $\Gamma'_{Reg}(R)$ contains an induced subgraph that is isomorphic to the house graph.
\end{lemma}
\begin{proof}
    Let $R \cong F_1 \times F_2 \times \dots \times F_n$ be a ring, where $F_1, F_2, \dots, F_n$ are fields. If $\text{maks}\{|F_1|,|F_2|,\dots,|F_n|\}>3$, then there exists $1 \leq i \leq n$ such that $|F_i|>3$. Consequently, there exists $0 \neq a_i \in F_i$ such that $a_i^{-1} \neq a_i$. Since $n \geq 2$, one can choose $1 \leq j \leq n$ with $j \neq i$. Suppose that
        \begin{align*}
            \textbf{0}&=(0_{F_1},0_{F_2},\dots,0_{F_n}),\\
            \textbf{x}&=(x_1,x_2,\dots,x_n), \text{ where } x_c=\begin{cases}
            1_{F_j}, &\text{if } c=j\\
            0_{F_c}, &\text{if } c \neq j
        \end{cases} \text{ for } 1 \leq c \leq n,\\
        \textbf{y}&=(y_1,y_2,\dots,y_n), \text{ where } y_c=\begin{cases}
            1_{F_j}, &\text{if } c=j\\
            a_i, &\text{if } c=i\\
            0_{F_c}, &\text{if } c \notin \{i,j\}
        \end{cases} \text{ for } 1 \leq c \leq n,\\
        \textbf{z}&=(z_1,z_2,\dots,z_n), \text{ where } z_c=\begin{cases}
            1_{F_j}, &\text{if } c=j\\
            a_i^{-1}, &\text{if } c=i\\
            0_{F_c}, &\text{if } c \notin \{i,j\}
        \end{cases} \text{ for } 1 \leq c \leq n,\\
        \textbf{u}&=(u_1,u_2,\dots,u_n), \text{ where } u_c=\begin{cases}
            a_i, &\text{if } c=i\\
            0_{F_c}, &\text{if } c\neq i
        \end{cases} \text{ for } 1 \leq c \leq n,\\
        \textbf{v}&=(v_1,v_2,\dots,v_n), \text{ where } v_c=\begin{cases}
            a_i^{-1}, &\text{if } c=i\\
            0_{F_c}, &\text{if } c\neq i
        \end{cases} \text{ for } 1 \leq c \leq n.
        \end{align*}
        Let the set $A=\{\textbf{x}, \textbf{y}, \textbf{z}, \textbf{u}, \textbf{v}\}$ be defined. 
        Thus, $E(\langle A \rangle_{\Gamma'_{Reg}(R)})=\{\textbf{xy}, \textbf{xz}, \textbf{yz}, \textbf{} \textbf{uz}, \textbf{vy}, \textbf{uv}\}$. Consequently, the subgraph induced by $A$ in the graph $\Gamma'_{Reg}(R)$ is isomorphic to the house graph.
\end{proof}

The preceding lemmas describe both the structural consequences of the presence of a house graph and the conditions that ensure its existence in $\Gamma'_{Reg}(R)$. By combining these observations, one can determine precisely when such an induced subgraph does not occur. The following theorem provides a complete characterization of rings for which $\Gamma'_{Reg}(R)$ is free of induced subgraphs isomorphic to the house graph.

\begin{theorem}\label{teo_subgrygtdkpunyarumah}
    Let $R \cong F_1 \times F_2 \times \dots \times F_n$, where $F_1, F_2, \dots, F_n$ are fields and $n \geq 2$. The graph $\Gamma'_{Reg}(R)$ does not contain an induced subgraph isomorphic to the house graph if and only if $R \cong \mathbb{F}_2 \times \mathbb{F}_2$, or $R \cong \mathbb{F}_2 \times \mathbb{F}_3$, or $R \cong \mathbb{F}_3 \times \mathbb{F}_3$, or $R \cong \mathbb{F}_2 \times \mathbb{F}_2 \times \mathbb{F}_2$. 
\end{theorem}
\begin{proof}
    Let $R \cong F_1 \times F_2 \times \dots \times F_n$, where $F_1, F_2, \dots, F_n$ are fields and $n \geq 2$.
    \begin{enumerate}
        \item[$(\Longleftarrow)$] If $R \cong \mathbb{F}_2 \times \mathbb{F}_2$, then $|V(\Gamma'_{Reg}(R))|=3<5$, so the graph $\Gamma'_{Reg}(R)$ does not contain an induced subgraph isomorphic to the house graph. If $R \cong \mathbb{F}_2 \times \mathbb{F}_3$, then the graph $\Gamma'_{Reg}(R)$ is isomorphic to the graph $P_5$, and hence it does not contain an induced subgraph isomorphic to the house graph since every induced subgraph of order $5$ is necessarily a path. If $R \cong \mathbb{F}_3 \times \mathbb{F}_3$, then the graph $\Gamma'_{Reg}(R)$ is isomorphic to the graph $C_8$. By Lemma \ref{lemmasubgrafrumahCnWm}, the graph $\Gamma'_{Reg}(R)$ does not contain an induced subgraph isomorphic to the house graph. If $R \cong \mathbb{F}_2 \times \mathbb{F}_2 \times \mathbb{F}_2$, then the graph $\Gamma'_{Reg}(R)$ is isomorphic to the graph $W_6$. By Lemma \ref{lemmasubgrafrumahCnWm}, the graph $\Gamma'_{Reg}(R)$ does not contain an induced subgraph isomorphic to the house graph. 
        \item[$(\Longrightarrow)$] If $R \ncong \mathbb{F}_2 \times \mathbb{F}_3$ and $R \ncong \mathbb{F}_3 \times \mathbb{F}_3$ and $R \ncong \mathbb{F}_2 \times \mathbb{F}_2 \times \mathbb{F}_2$, then we consider the following cases.
        \begin{enumerate}
            \item Case $F_i \cong \mathbb{F}_2$ for every $i=1,2,\dots,n$, with $n \geq 4$. There exists a set $A = \{a_1,a_2,a_3,a_4,a_5\} \subseteq V(\Gamma'_{Reg}(R))$, with 
            \begin{align*}
                &a_1=(1,0,0,\dots,0,0,0), \text{ } a_2=(1,0,0,\dots,0,1,1), \text{ } a_3=(1,1,0,\dots,0,0,1), \\&a_4=(0,0,0,\dots,0,0,1),\text{ } a_5=(0,1,0,\dots,0,0,1).
             \end{align*} 
            Consequently, $E(\langle A \rangle_{\Gamma'_{Reg}(R)})=\{a_1a_2, a_1a_3, a_2a_4, a_3a_4, a_3a_5, a_4a_5\}$. Hence, the subgraph induced by $A$ in the graph $\Gamma'_{Reg}(R)$ is isomorphic to the house graph.
            \item Case $n \geq 3$ and there exists $1 \leq i \leq n$ such that $F_i \cong \mathbb{F}_3$ and $|F_j| \leq 3$ for every $j=1,2,\dots,n$, then $R \cong F_i \times \prod_{1\leq j \leq n, j \neq i}F_j$. There exists a set $ B = \{b_1,b_2,b_3,b_4,b_5\} \subseteq V(\Gamma'_{Reg}(R))$, with 
            \begin{align*}
                &b_1=(2,1,1,0,\dots,0), \quad b_2=(0,0,1,0,\dots,0), \quad b_3=(0,1,0,0,\dots,0),\\ &b_4=(1,1,1,0,\dots,0),\quad b_5=(1,1,0,0,\dots,0).
            \end{align*} 
            Consequently, $E(\langle B \rangle_{\Gamma'_{Reg}(R)})=\{b_1b_2, b_1b_3, b_2b_4, b_3b_4, b_3b_5, b_4b_5\}$. Thus, the subgraph induced by $B$ in the graph $\Gamma'_{Reg}(R)$ is isomorphic to the house graph.
            \item If there exists $1 \leq i \leq n$ such that $|F_i|>3$, then $\text{maks}\{|F_1|,|F_2|,\dots,|F_n|\}>3$. Based on Lemma \ref{lemmamaksFgrafrumah}, the graph $\Gamma'_{Reg}(R)$ contains an induced subgraph that is isomorphic to the house graph. 
        \end{enumerate}
    \end{enumerate}
\end{proof}

The previous theorem identifies all rings for which $\Gamma'_{Reg}(R)$ avoids an induced subgraph isomorphic to the house graph. This classification naturally leads to a more precise description of the overall structure of the graphs. In particular, it becomes possible to determine exactly when $\Gamma'_{Reg}(R)$ belongs to several well-known families of graphs. The following theorem provides such a characterization.

\begin{theorem}
    Let $R$ be a ring. The following statements hold.
    \begin{enumerate}
        \item The graph $\Gamma'_{Reg}(R)$ is a path graph if and only if $R \cong \mathbb{F}_2 \times \mathbb{F}_2$ or $R \cong \mathbb{F}_2 \times \mathbb{F}_3$.
        \item The graph $\Gamma'_{Reg}(R)$ is a cycle graph if and only if $R \cong \mathbb{F}_3 \times \mathbb{F}_3$.
        \item The graph $\Gamma'_{Reg}(R)$ is a wheel graph if and only if $R \cong \mathbb{F}_2 \times \mathbb{F}_2 \times \mathbb{F}_2$.
    \end{enumerate}
\end{theorem}
\begin{proof}
    Based on Lemma \ref{lemmasubgrafrumahCnWm}, if the graph $\Gamma'_{Reg}(R) \cong P_k$ for some natural number $k$, or $\Gamma'_{Reg}(R) \cong C_n$ for some natural number $n$, or $\Gamma'_{Reg}(R) \cong W_m$ for some natural number $m$, then the graph $\Gamma'_{Reg}(R)$ does not contain an induced subgraph that is isomorphic to the house graph. Using Theorem \ref{teo_subgrygtdkpunyarumah}, it follows that $R \cong \mathbb{F}_2 \times \mathbb{F}_2$, or $R \cong \mathbb{F}_2 \times \mathbb{F}_3$, or $R \cong \mathbb{F}_3 \times \mathbb{F}_3$, or $R \cong \mathbb{F}_2 \times \mathbb{F}_2 \times \mathbb{F}_2$. In this case, the graph $\Gamma'_{Reg}(\mathbb{F}_2 \times \mathbb{F}_2) \cong P_3$, the graph $\Gamma'_{Reg}(\mathbb{F}_2 \times \mathbb{F}_3) \cong P_5$, the graph $\Gamma'_{Reg}(\mathbb{F}_3 \times \mathbb{F}_3) \cong C_8$, and the graph $\Gamma'_{Reg}(\mathbb{F}_2 \times \mathbb{F}_2 \times \mathbb{F}_2) \cong W_6$.
\end{proof}

To further understand the structural complexity of $\Gamma'_{Reg}(R)$, it is useful to examine its planarity. In particular, certain properties of the component fields in the decomposition of $R$ can force the graph to be non-planar. The following lemma provides a condition under which this occurs.

\begin{lemma}
    Let $R \cong F_1 \times F_2 \times \dots \times F_n$, where $F_1, F_2, \dots, F_n$ are fields and $n \geq 2$. If there exist $F_i, F_j$ with $1 \leq i <j \leq n$ such that $|U''(F_i)|,|U''(F_j)| >0$, then $\Gamma'_{Reg}(R)$ is a non-planar graph.
\end{lemma}
\begin{proof}
    If there exist $F_i, F_j$ with $1 \leq i <j \leq n$ such that $|U''(F_i)|,|U''(F_j)| >0$, then there exist $a,a^{-1} \in U''(F_i)$ and $b,b^{-1} \in U''(F_j)$ with $a \neq a^{-1}$ and $b \neq b^{-1}$. Since the order of the factors in the direct product of rings does not affect the structure, as any reordering yields an isomorphic ring, we may assume that $i=1$ and $j=2$. Consider the vertices in the set
    \begin{align*}
        X&=\{(0_{\mathbb{F}_1},b,0_{\mathbb{F}_3},\dots,0_{\mathbb{F}_n}),(0_{\mathbb{F}_1},b^{-1},,0_{\mathbb{F}_3},\dots,0_{\mathbb{F}_n}),(a,0_{\mathbb{F}_2},0_{\mathbb{F}_3},\dots,0_{\mathbb{F}_n}), \\& \quad \quad (a^{-1},0_{\mathbb{F}_2},0_{\mathbb{F}_3},\dots,0_{\mathbb{F}_n}),(a,b,0_{\mathbb{F}_3},\dots,0_{\mathbb{F}_n}),(a^{-1},b,0_{\mathbb{F}_3},\dots,0_{\mathbb{F}_n}), \\& \quad \quad (a,b^{-1},0_{\mathbb{F}_3},\dots,0_{\mathbb{F}_n}),(a^{-1},b^{-1},0_{\mathbb{F}_3},\dots,0_{\mathbb{F}_n})\}.
    \end{align*}
    The subgraph induced by $X$ in the graph $\Gamma'_{Reg}(R)$ is shown in the following illustration.
    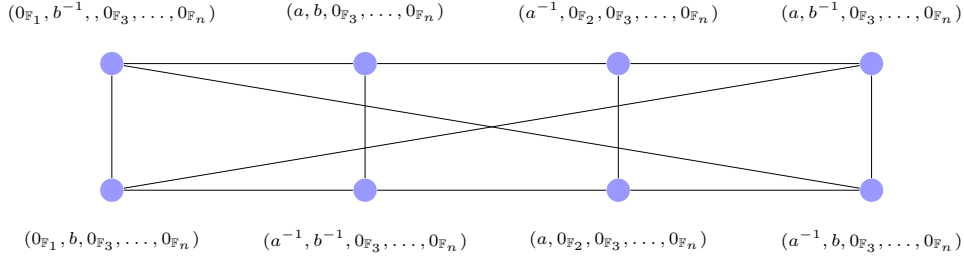
\begin{figure}[H]
    \begin{center}
        \resizebox{0.8\textwidth}{!}{\begin{tikzpicture}  
				[scale=.9,auto=center,roundnode/.style={circle,fill=blue!40}]
				\node[roundnode] (a1) at (-6,1) {};  
				\node[roundnode] (a2) at (-2,1)  {};  
				\node[roundnode] (a3) at (2,1)  {};
                \node[roundnode] (a4) at (6,1)  {};
				\node[roundnode] (b1) at (-6,-1) {};
				\node[roundnode] (b2) at (-2,-1)  {};
				\node[roundnode] (b3) at (2,-1) {}; 
                \node[roundnode] (b4) at (6,-1) {}; 
                
                \begin{scriptsize}
                    \draw[color=black] (-6,1.8) node {$(0_{\mathbb{F}_1},b^{-1},,0_{\mathbb{F}_3},\dots,0_{\mathbb{F}_n})$};
                    \draw[color=black] (-2,1.8) node {$(a,b,0_{\mathbb{F}_3},\dots,0_{\mathbb{F}_n})$};
                    \draw[color=black] (2,1.8) node {$(a^{-1},0_{\mathbb{F}_2},0_{\mathbb{F}_3},\dots,0_{\mathbb{F}_n})$};
                    \draw[color=black] (6,1.8) node {$(a,b^{-1},0_{\mathbb{F}_3},\dots,0_{\mathbb{F}_n})$};%
                    \draw[color=black] (-6,-1.8) node {$(0_{\mathbb{F}_1},b,0_{\mathbb{F}_3},\dots,0_{\mathbb{F}_n})$};
                    \draw[color=black] (-2,-1.8) node {$(a^{-1},b^{-1},0_{\mathbb{F}_3},\dots,0_{\mathbb{F}_n})$};
                    \draw[color=black] (2,-1.8) node {$(a,0_{\mathbb{F}_2},0_{\mathbb{F}_3},\dots,0_{\mathbb{F}_n})$};
                    \draw[color=black] (6,-1.8) node {$(a^{-1},b,0_{\mathbb{F}_3},\dots,0_{\mathbb{F}_n})$};
                \end{scriptsize} 
                \draw (a1) -- (a2);
                \draw (a2) -- (a3);
				\draw (a3) -- (a4);
                \draw (b1) -- (b2);
                \draw (b2) -- (b3);
				\draw (b3) -- (b4);
                \draw (a1) -- (b1);
                \draw (a2) -- (b2);
				\draw (a3) -- (b3);
                \draw (a4) -- (b4);
                \draw (a1) -- (b4);
				\draw (b1) -- (a4);
			\end{tikzpicture}}\caption{A subgraph of the graph $\Gamma'_{Reg}(R)$}\label{subgrafGamReg1}
		\end{center}
\end{figure}
This subgraph contains a subgraph that is homeomorphic to the following graph $K_{3,3}$.
                
\end{proof}

The condition given in the lemma can be expressed more explicitly in terms of the sizes of the component fields. In particular, it leads to a simpler criterion for detecting non-planarity of $\Gamma'_{Reg}(R)$. The following corollary states this consequence.

\begin{corollary}
    Let $R \cong F_1 \times F_2 \times \dots \times F_n$, where $F_1, F_2, \dots, F_n$ are fields and $n \geq 2$. If there exist $F_i, F_j$ with $1 \leq i <j \leq n$ such that $F_i,F_j \ncong \mathbb{F}_2$ and $F_i,F_j \ncong \mathbb{F}_3$, then $\Gamma'_{Reg}(R)$ is a non-planar graph.
\end{corollary}

To relate the structure of $\Gamma'_{Reg}(R)$ more closely to its algebraic components, it is useful to focus on the subgraph induced by the idempotent elements of $R$. This perspective reveals a simpler underlying pattern that can be described in terms of products of the field $\mathbb{F}_2$. The following theorem formalizes this observation.

\begin{theorem}
    Let $R \cong F_1 \times F_2 \times \dots \times F_n$, where $F_1,F_2,\dots,F_n$ are fields and $n \geq 1$. The subgraph induced by $Id(R) \setminus \{0_R\}$ of the graph $\Gamma'_{Reg}(R)$ is isomorphic to the graph $\Gamma'_{Reg}((\mathbb{F}_2)^n)$.
\end{theorem}
\begin{proof}
    It is known that $Id(R) \setminus \{0_R\}=\{(a_1,a_2,\dots,a_n) \mid a_i \in \{0_{F_i}, 1_{F_i}\}\} \setminus \{(0_{F_1}, 0_{F_2}, \dots, 0_{F_n})\}$. Define a bijection $\phi: Id(R) \setminus \{0_R\} \to (\mathbb{F}_2)^n \setminus \{0_{(\mathbb{F}_2)^n}\}$, by $\phi((a_1,a_2,\dots,a_n))=(f_1(a_1),f_2(a_2),\dots,f_n(a_n))$ where $f_i(a_i)=\begin{cases}
        0_{\mathbb{F}_2}, &\text{if } a_i=0_{F_i}\\
        1_{\mathbb{F}_2}, &\text{if } a_i=1_{F_i}
    \end{cases}$ for each $1 \leq i \leq n$. Throughout this work, for any field $F$, we use the notation $\delta(0_F)=0 \in \mathbb{Z}$ and $\delta(1_F)=1 \in \mathbb{Z}$.
    For any distinct vertices $(a_1,a_2,\dots,a_n), (b_1,b_2,\dots,b_n) \in Id(R) \setminus \{0_R\}$, we have
    \begin{align*}
        &(a_1,a_2,\dots,a_n)(b_1,b_2,\dots,b_n) \in E(\Gamma'_{Reg}(R))\\
        &\iff a_ib_ia_i=a_i \text{ for each } 1 \leq i \leq n, \text{ or } b_ia_ib_i=b_i \text{ for each } 1 \leq i \leq n\\
        &\iff \delta(a_i) \leq \delta(b_i) \text{ for each } 1 \leq i \leq n, \text{ or } \delta(b_i) \leq \delta(a_i) \text{ for each } 1 \leq i \leq n\\
        &\iff \delta(f_i(a_i)) \leq \delta(f_i(b_i)) \text{ for each } 1 \leq i \leq n, \\
        & \qquad \text{ or } \delta(f_i(b_i)) \leq \delta(f_i(a_i)) \text{ for each } 1 \leq i \leq n\\
        &\iff f_i(a_i)f_i(b_i)f_i(a_i)=f_i(a_i) \text{ for each } 1 \leq i \leq n, \\& \qquad \text{ or } f_i(b_i)f_i(a_i)f_i(b_i)=f_i(b_i) \text{ for each } 1 \leq i \leq n\\
        &\iff \phi((a_1,a_2,\dots,a_n))\phi((b_1,b_2,\dots, b_n)) \in E(\Gamma'_{Reg}((\mathbb{F}_2)^n)).
    \end{align*}
    Thus, $\phi$ is a graph isomorphism. In other words, $\langle Id(R) \setminus \{0_R\} \rangle_{\Gamma'_{Reg}(R)} \cong \Gamma'_{Reg}((\mathbb{F}_2)^n)$.
\end{proof}

To determine the structure of $\Gamma'_{Reg}(R)$ for an arbitrary finite commutative von Neumann regular ring $R$, we begin by examining how the graph behaves under direct products with fields. This approach allows the structure to be built step by step from simpler components. The following theorem provides a key construction that describes $\Gamma'_{Reg}(R \times F)$ in terms of $\Gamma'_{Reg}(R \times \mathbb{F}_2)$ and certain graph operations.

\begin{theorem}\label{metodeF2keF}
    Let $R$ be a finite commutative regular ring. Let $G=\Gamma'_{Reg}(R \times \mathbb{F}_2)$, $A=\{(a,1_{\mathbb{F}_2}) \mid a \in R, a^3=a\}$, and $J=\{(u,0_{\mathbb{F}_2}) \mid u \in R^*\}$, where $R^*=R \setminus \{0_R\}$. For any field $F$ with $|F|=n=p^k$, where $p$ is a prime number and $k$ is a positive integer, the following holds
    \begin{align*}
        \Gamma'_{Reg}(R \times F) \cong \begin{cases}
            \bigsqcup(G, \underbrace{AmB_A(G \mid J), \dots ,AmB_A(G \mid J)}_{\frac{n-2}{2}} \mid J) , &\text{if } p=2\\
            \bigsqcup(G, G, \underbrace{AmB_A(G \mid J), \dots ,AmB_A(G \mid J)}_{\frac{n-3}{2}} \mid J) ,  &\text{otherwise}.
        \end{cases}
    \end{align*}
\end{theorem}

\begin{proof}
    Let $R$ be a finite commutative regular ring. It is known that $V(G)=V(\Gamma'_{Reg}(R \times \mathbb{F}_2))=\{(u,0_{\mathbb{F}_2}) \mid u \in R^*\} \cup \{(u,1_{\mathbb{F}_2}) \mid u \in R\}$. Take any field $F$ with $|F|=n=p^k$, where $p$ is a prime number and $k$ is a positive integer. Let $H=\Gamma'_{Reg}(R \times F)$, $U_0=\{(u,0_F) \mid u \in R^*\}$, and $U''(F)=\{b_i, c_i \mid 1 \leq i \leq \tfrac{n-|U'(F)|-1}{2}\}$ with $b_ic_i=1_F$ for each $1 \leq i \leq \tfrac{n-|U'(F)|-1}{2}$. Define $K_a = \{(u,a) \mid u \in R\}$ for each $a \in U'(F)$ and $L_i=\{(u,b_i),(u,c_i) \mid u \in R\}$ for each $1 \leq i \leq \tfrac{n-|U'(F)|-1}{2}$. We will show that $\langle U_0 \cup K_a \rangle_H \cong G$ for each $a \in U'(F)$, and $\langle U_0 \cup L_i \rangle_H \cong AmB_A(G \mid J)$ for each $1 \leq i \leq \tfrac{n-|U'(F)|-1}{2}$. Furthermore, it remains to show that for every $x \in X$ and $y \in Y$ with $X,Y \in \{K_a, L_i \mid a \in U'(F), 1 \leq i \leq \tfrac{n-|U'(F)|-1}{2}\}$ and $X \neq Y$, we have $xy \notin E(\Gamma'_{Reg}(R \times F))$. 

    \begin{enumerate}
        \item First, for any $a \in U'(F)$, define $\psi: U_0 \cup K_a \to V(G)$ by $\psi((u,0_F))=(u,0_{\mathbb{F}_2})$ and $\psi((u,a))=(u,1_{\mathbb{F}_2})$. Then $\psi$ is a bijection and for every $(u,x),(v,y) 
        \in U_0 \cup K_a$ with $\psi((u,x))=(u,x')$ and $\psi((v,y))=(v,y')$, we have
        \begin{align*}
            (u,x)(v,y) \in &E(\langle U_0 \cup K_a \rangle_H) \iff (u,x)(v,y) \in E(H)\\
            &\iff (uvu=u \wedge xyx=x) \vee (vuv=v \wedge yxy=y)\\
            &\iff (uvu=u \wedge (x,y) \in \{(0_F,0_F),(0_F,a),(a,a)\}) \\
            & \qquad \vee (vuv=v \wedge (y,x) \in \{(0_F,0_F),(0_F,a),(a,a)\})\\
            &\iff (uvu=u \wedge (x',y') \in \{(0_{\mathbb{F}_2},0_{\mathbb{F}_2}),(0_{\mathbb{F}_2},1_{\mathbb{F}_2}),(1_{\mathbb{F}_2},1_{\mathbb{F}_2})\}) \\
            & \qquad \vee (vuv=v \wedge (y,x) \in \{(0_{\mathbb{F}_2},0_{\mathbb{F}_2}),(0_{\mathbb{F}_2},1_{\mathbb{F}_2}),(1_{\mathbb{F}_2},1_{\mathbb{F}_2})\})\\
            &\iff (uvu=u \wedge x'y'x'=x') \vee (vuv=v \wedge y'x'y'=y')\\
            &\iff \psi((u,x))\psi((v,y))=(u,x')(v,y') \in E(G).
        \end{align*}
        Hence, $\psi$ is a graph isomorphism. Therefore, $\langle U_0 \cup K_a \rangle_H \cong G$.

        \item Second, for each $1 \leq i \leq \tfrac{n-|U'(F)|-1}{2}$, define $\phi: U_0 \cup L_i \to V(AmB_A(G \mid J))$ by $\phi((u,0_F))=(u,0_{\mathbb{F}_2})$, $\phi((u,b_i))=((u,1_{\mathbb{F}_2}),0)$, and $\phi((u,c_i))=((u,1_{\mathbb{F}_2}),1)$. Observe that
        \begin{align*}
            E(AmB_A(G \mid J))=E(B(G-J)) \cup W \cup  E(\langle J \rangle_G) \cup T
        \end{align*}
        where
        \begin{align*}
            E(B(G-J))&=\{((u,1_{\mathbb{F}_2}),0)((v,1_{\mathbb{F}_2}),1) \mid (u,1_{\mathbb{F}_2})(v,1_{\mathbb{F}_2}) \in E(G)\}\\
            &=\{((u,1_{\mathbb{F}_2}),0)((v,1_{\mathbb{F}_2}),1) \mid  u \neq v,uvu=u \vee vuv=v\}\\
            E(\langle J \rangle_G)&=\{(u,0_{\mathbb{F}_2})(v,0_{\mathbb{F}_2}) \mid uvu=u \vee vuv=v\}\\
            W &=\{((v,1_{\mathbb{F}_2}),0)((v,1_{\mathbb{F}_2}),1) \mid (v,1_{\mathbb{F}_2}) \in A\}\\
            &=\{((v,1_{\mathbb{F}_2}),0)((v,1_{\mathbb{F}_2}),1) \mid v^3=v\}\\
            T&=\{(u,0_{\mathbb{F}_2})((v,1_{\mathbb{F}_2}),0),(u,0_{\mathbb{F}_2})((v,1_{\mathbb{F}_2}),1) \mid\\ 
            & \qquad \quad (u,0_{\mathbb{F}_2}) \in J, (v,1_{\mathbb{F}_2}) \in V(G-J), (u,0_{\mathbb{F}_2})(v,1_{\mathbb{F}_2}) \in E(G)\}\\
            &=\{(u,0_{\mathbb{F}_2})((v,1_{\mathbb{F}_2}),0),(u,0_{\mathbb{F}_2})((v,1_{\mathbb{F}_2}),1) \mid u,v \in R, uvu=u\}.
        \end{align*}
        It follows that $\phi$ is a bijection and for every $(u,x),(v,y) 
        \in U_0 \cup L_i$ we have
        \begin{align*}
            &(u,x)(v,y) \in E(\langle U_0 \cup L_i \rangle_H) \iff (u,x)(v,y) \in E(H)\\
            &\iff (uvu=u \wedge xyx=x) \vee (vuv=v \wedge yxy=y)\\
            &\iff (uvu=u \wedge (x,y) \in \{(0_F,0_F),(0_F,b_i),(0_F,c_i),(b_i,c_i),(c_i,b_i)\}) \\
            & \qquad \quad \vee (vuv=v \wedge(y,x) \in \{(0_F,0_F),(0_F,b_i),(0_F,c_i),(b_i,c_i),(c_i,b_i)\})\\
            &\iff (uvu=u \vee vuv=v) \wedge (x,y)=(0_F,0_F)\\
            & \qquad \quad \vee (uvu=u \wedge (x,y) \in \{(0_F,b_i),(0_F,c_i),(b_i,c_i),(c_i,b_i)\}) \\
            & \qquad \quad \vee (vuv=v \wedge(y,x) \in \{(0_F,b_i),(0_F,c_i),(b_i,c_i),(c_i,b_i)\})\\
            &\iff [(uvu=u \vee vuv=v) \wedge (x,y)=(0_F,0_F)]\\
            & \qquad \quad \vee [(uvu=u \vee vuv=v) \wedge (x,y) \in \{(b_i,c_i),(c_i,b_i)\}]\\
            & \qquad \quad \vee (uvu=u \wedge (x,y) \in \{(0_F,b_i),(0_F,c_i)\}) \\
            & \qquad \quad \vee (vuv=v \wedge(y,x) \in \{(0_F,b_i),(0_F,c_i)\}).
        \end{align*}
        Consider the following cases.
        \begin{enumerate}
            \item The condition $uvu=u$ or $vuv=v$, and $(x,y)=(0_F,0_F)$ holds if and only if $\phi((u,x))=(u,0_{\mathbb{F}_2})$ and $\phi((v,y))=(v,0_{\mathbb{F}_2})$ with $uvu=u$ or $vuv=v$. This is equivalent to $\phi((u,x))\phi((v,y)) \in E(\langle J \rangle_G)$.
            \item The condition $uvu=u$ or $vuv=v$, and $(x,y) \in \{(b_i,c_i),(c_i,b_i)\}$ holds if and only if $\phi((u,x))=((u,1_{\mathbb{F}_2}),0)$ and $\phi((v,y))=((v,1_{\mathbb{F}_2}),1)$, or $\phi((u,x))=((u,1_{\mathbb{F}_2}),1)$ and $\phi((v,y))=((v,1_{\mathbb{F}_2}),0)$ with $uvu=u$ or $vuv=v$. This is equivalent to $\phi((u,x))\phi((v,y)) \in E(B(G-J)) \cup W$.
            \item If $uvu=u$ and $(x,y) \in \{(0_F,b_i),(0_F,c_i)\}$, then $\phi((u,x))=(u,0_{\mathbb{F}_2})$, and $\phi((v,y))=((v,1_{\mathbb{F}_2}),0)$ or $\phi((v,y))=((v,1_{\mathbb{F}_2}),1)$, with $uvu=u$. Hence, $\phi((u,x))\phi((v,y)) \in T$. Similarly, if $vuv=v$ and $(y,x) \in \{(0_F,b_i),(0_F,c_i)\}$, then $\phi((u,x))=((u,1_{\mathbb{F}_2}),0)$ or $\phi((u,x))=((u,1_{\mathbb{F}_2}),1)$, and $\phi((v,y))=(v,0_{\mathbb{F}_2})$, with $vuv=v$. Hence, $\phi((u,x))\phi((v,y)) \in T$.
            \item If $\phi((u,x))\phi((v,y)) \in T$, then there are two possibilities. First, $\phi((u,x))=(u,0_{\mathbb{F}_2})$, and $\phi((v,y))=((v,1_{\mathbb{F}_2}),0)$ or $\phi((v,y))=((v,1_{\mathbb{F}_2}),1)$, with $uvu=u$. Second, $\phi((u,x))=((u,1_{\mathbb{F}_2}),0)$ or $\phi((u,x))=((u,1_{\mathbb{F}_2}),1)$, and $\phi((v,y))=(v,0_{\mathbb{F}_2})$, with $vuv=v$.
        \end{enumerate}
        Therefore, $(u,x)(v,y) \in E(\langle U_0 \cup L_i \rangle_H)$ if and only if $\phi(u,x)\phi(v,y) \in E(AmB_A(G \mid J))$, and hence $\phi$ is a graph isomorphism. Thus, $\langle U_0 \cup L_i \rangle_H \cong AmB_A(G \mid J)$.

        \item Third, take any $x \in X$ and $y \in Y$ with $X,Y \in \{K_a, L_i \mid a \in U'(F), 1 \leq i \leq \tfrac{n-|U'(F)|-1}{2}\}$ and $X \neq Y$. Then $\{x,y\} \in \{\{(u,p),(v,q)\} \mid u,v \in R, p,q \in U(F), pq \neq 1_F\}$. Clearly, $pqp \neq p \iff pq \neq 1_F \iff qpq \neq q$. Hence, $xyx \neq x$ and $yxy \neq y$. Therefore, $xy \notin E(\Gamma'_{Reg}(R \times F))$.
    \end{enumerate}

    Therefore, we obtain
    \begin{align*}
        \Gamma'_{Reg}(R \times F) \cong \bigsqcup\{\underbrace{G, \dots, G}_{|U'(F)|}, \underbrace{AmB_A(G \mid J), \dots ,AmB_A(G \mid J)}_{\frac{n-|U'(F)|-1}{2}} \mid J \}.
    \end{align*}

    In this case, if $p=2$, then $|U'(F)|=1$ and $|U''(F)|=n-2$. If $p \neq 2$, then $|U'(F)|=2$ and $|U''(F)|=n-3$. Hence,
    \begin{align*}
        \Gamma'_{Reg}(R \times F) \cong \begin{cases}
            \bigsqcup(G, \underbrace{AmB_A(G \mid J), \dots ,AmB_A(G \mid J)}_{\frac{n-2}{2}} \mid J) , &\text{if } p=2\\
            \bigsqcup(G, G, \underbrace{AmB_A(G \mid J), \dots ,AmB_A(G \mid J)}_{\frac{n-3}{2}} \mid J) ,  &\text{otherwise}.
        \end{cases}
    \end{align*}
\end{proof}

We will consider some examples here.
    The structure of the graph $\Gamma'_{Reg}(\mathbb{F}_2 \times \mathbb{F}_3 \times \mathbb{F}_5)$ can be obtained through the following algorithm:
    \begin{enumerate}
        \item Let the graph $G_1 = \Gamma'_{Reg}(\mathbb{F}_2 \times \mathbb{F}_2 \times \mathbb{F}_2)$. 
        \item Let $J_1=\{(0_{\mathbb{F}_2},0_{\mathbb{F}_2},1_{\mathbb{F}_2}),(1_{\mathbb{F}_2},0_{\mathbb{F}_2},0_{\mathbb{F}_2}),(1_{\mathbb{F}_2},0_{\mathbb{F}_2},1_{\mathbb{F}_2})\}$. Construct the graph $G_2=\bigsqcup(G_1,G_1 \mid J_1)$.
        \item Let $J_2=\{(1_{\mathbb{F}_2}, 0_{\mathbb{F}_3}, 0_{\mathbb{F}_2}), (1_{\mathbb{F}_2}, 1_{\mathbb{F}_3}, 0_{\mathbb{F}_2}), (1_{\mathbb{F}_2}, 2_{\mathbb{F}_3}, 0_{\mathbb{F}_2}),(0_{\mathbb{F}_2}, 1_{\mathbb{F}_3}, 0_{\mathbb{F}_2}), $ $(0_{\mathbb{F}_2}, 2_{\mathbb{F}_3}, 0_{\mathbb{F}_2})\}$, and \\$A=\{(1_{\mathbb{F}_2}, 0_{\mathbb{F}_3}, 1_{\mathbb{F}_2}), (1_{\mathbb{F}_2}, 1_{\mathbb{F}_3}, 1_{\mathbb{F}_2}), (1_{\mathbb{F}_2}, 2_{\mathbb{F}_3}, 1_{\mathbb{F}_2}),(0_{\mathbb{F}_2}, 1_{\mathbb{F}_3}, 1_{\mathbb{F}_2}), (0_{\mathbb{F}_2}, 2_{\mathbb{F}_3}, 1_{\mathbb{F}_2}), (0_{\mathbb{F}_2}, 0_{\mathbb{F}_3}, 1_{\mathbb{F}_2})\}$. Construct the graph $G_3=\bigsqcup(G_2,G_2, AmB_A(G_2 \mid J_2) \mid J_2)$.
        \item The graph $\Gamma'_{Reg}(\mathbb{F}_2 \times \mathbb{F}_3 \times \mathbb{F}_5)$ is isomorphic to $G_3$.
    \end{enumerate}

    The structure of the graph $\Gamma'_{Reg}(\mathbb{F}_3 \times \mathbb{F}_3 \times \mathbb{F}_3)$ can be obtained through the following algorithm:
    \begin{enumerate}
        \item Let the graph $G_1 = \Gamma'_{Reg}(\mathbb{F}_2 \times \mathbb{F}_2 \times \mathbb{F}_2)$. 
        \item Let $J_1=\{(0_{\mathbb{F}_2},0_{\mathbb{F}_2},1_{\mathbb{F}_2}),(0_{\mathbb{F}_2},1_{\mathbb{F}_2},0_{\mathbb{F}_2}),(0_{\mathbb{F}_2},1_{\mathbb{F}_2},1_{\mathbb{F}_2})\}$. Construct the graph $G_2=\bigsqcup(G_1,G_1 \mid J_1)$.
        \item Let $J_2=\{(1_{\mathbb{F}_3}, 0_{\mathbb{F}_2}, 0_{\mathbb{F}_2}), (2_{\mathbb{F}_3}, 0_{\mathbb{F}_2}, 0_{\mathbb{F}_2}), (0_{\mathbb{F}_3}, 0_{\mathbb{F}_2}, 1_{\mathbb{F}_2}),(1_{\mathbb{F}_3}, 0_{\mathbb{F}_2}, 1_{\mathbb{F}_2}), $ $(2_{\mathbb{F}_3}, 0_{\mathbb{F}_2}, 1_{\mathbb{F}_2})\}$. Construct the graph $G_3=\bigsqcup(G_2,G_2 \mid J_2)$.
        \item Let $J_3=\{(1_{\mathbb{F}_3}, 0_{\mathbb{F}_3}, 0_{\mathbb{F}_2}), (2_{\mathbb{F}_3}, 0_{\mathbb{F}_3}, 0_{\mathbb{F}_2}), (0_{\mathbb{F}_3}, 1_{\mathbb{F}_3}, 0_{\mathbb{F}_2}),(0_{\mathbb{F}_3}, 2_{\mathbb{F}_3}, 0_{\mathbb{F}_2}), $ $(1_{\mathbb{F}_3}, 1_{\mathbb{F}_3}, 0_{\mathbb{F}_2}),(1_{\mathbb{F}_3}, 2_{\mathbb{F}_3}, 0_{\mathbb{F}_2}),$ $(2_{\mathbb{F}_3}, 1_{\mathbb{F}_3}, 0_{\mathbb{F}_2}), (2_{\mathbb{F}_3}, 2_{\mathbb{F}_3}, 0_{\mathbb{F}_2})\}$. Construct the graph $G_4=\bigsqcup(G_3,G_3 \mid J_3)$.
        \item The graph $\Gamma'_{Reg}(\mathbb{F}_3 \times \mathbb{F}_3 \times \mathbb{F}_3)$ is isomorphic to $G_4$.
    \end{enumerate}

The construction given in the theorem can be applied iteratively to describe the structure of $\Gamma'_{Reg}(R)$ for any finite commutative von Neumann regular ring. In particular, this approach reduces the problem to understanding the graph associated with $(\mathbb{F}_2)^n$. The following corollary summarizes this observation.

\begin{corollary}
    Let $R \cong F_1 \times F_2 \times \dots \times F_n$, where $F_1, F_2, \dots, F_n$ are fields and $n \geq 1$. The structure of the graph $\Gamma'_{Reg}(R)$ can be obtained from the graph $\Gamma'_{Reg}((\mathbb{F}_2)^n)$ by applying an algorithm analogous to the one described in the example.
\end{corollary}

In this setting, the structure of the graph $\Gamma'_{Reg}((\mathbb{F}_2)^n)$ plays a crucial role in the algorithm used to determine the structure of $\Gamma'_{Reg}(R)$ for an arbitrary finite commutative von Neumann regular ring. Indeed, for every $a \in (\mathbb{F}_2)^n$, we have $a^2 = a$, and hence $a1_{(\mathbb{F}_2)^n}a = a^2 = a$. This implies that each element $a$ is adjacent to $1_{(\mathbb{F}_2)^n}$ in the graph $\Gamma'_{Reg}((\mathbb{F}_2)^n)$. As a result, the graph can be decomposed into the join of the vertex $1_{(\mathbb{F}_2)^n}$ and the subgraph induced by the remaining vertices. The following theorem describes the relationship between the structure of $\Gamma'_{Reg}((\mathbb{F}_2)^n)$ and the inclusion ideal graph $In((\mathbb{F}_2)^n)$.
\begin{theorem}
    Let $R \cong (\mathbb{F}_2)^n$. The subgraph induced by $R \setminus \{0_R, 1_R\}$ of the graph $\Gamma'_{Reg}(R)$ is isomorphic to the graph $In(R)$.
\end{theorem}
\begin{proof}
     Let $R \cong (\mathbb{F}_2)^n$. Since every element in $R$ is idempotent, using Theorem 1.1 in \cite{goodearl}, it follows that each element of $R$ corresponds bijectively to an ideal of $R$. In this setting, the elements $0_{R}$ and $1_{R}$ correspond to the trivial ideals $\{0_{R}\}$ and $R$, respectively. Consequently, all elements in $R \setminus \{0_R, 1_R\}$ correspond bijectively to all nontrivial ideals of $R$. Thus, we can define a bijective function $\alpha: R \setminus \{0_R, 1_R\} \to X_R$, where $X_R$ denotes the set of all ideals of $R$, by $\alpha(a)=[a]$ for each $a \in R \setminus \{0_R, 1_R\}$. 

For arbitrary $a,b \in R \setminus \{0_R, 1_R\}$. Observe that $ab \in E(\Gamma'_{Reg}(R))$ if and only if $a \neq b \text{ and }( ab=aba=a \text{ or } ba=bab=b)$, then $[a] \subset [b] \text{ or } [b] \subset [a]$.
Conversely, if $[a] \subset [b]$, then $a \neq b$, and $a \in [b]$, so $a= br$ for some $r \in R$. It follows that $aba=b^3r^2=(br)^2=a^2=a$. Similarly, if $[b] \subset [a]$, then $a \neq b$ and $bab=b$. Therefore, $ab \in E(\Gamma'_{Reg}(R))$ if and only if $[a] \subset [b]$ or $[b] \subset [a]$, if and only if $\alpha(a)\alpha(b)=[a][b] \in E(In(R))$. Hence, $\alpha$ is a graph isomorphism, and the subgraph induced by $R \setminus \{0_R, 1_R\}$ of the graph $\Gamma'_{Reg}(R)$ is isomorphic to the graph $In(R)$.
\end{proof}
The previous theorem shows that the subgraph induced by $R \setminus \{0_R, 1_R\}$ in $\Gamma'_{Reg}(R)$ has the same structure as the inclusion ideal graph $In(R)$. It remains to consider the role of the element $1_R$ in $\Gamma'_{Reg}(R)$. By definition, $1_R$ is adjacent to every other vertex in $Reg(R) \setminus \{0_R\}$, which implies that it forms a universal vertex in the reduced graph. Consequently, $\Gamma'_{Reg}(R)$ can be obtained by taking the join of a single vertex with the graph $In(R)$, leading directly to the following corollary.

\begin{corollary}
    Let $R \cong (\mathbb{F}_2)^n$. Then $\Gamma'_{Reg}(R) \cong K_1 + In(R)$.
\end{corollary}

The results obtained above not only provide a structural description of $\Gamma'_{Reg}(R)$ but also reveal a close connection with other algebraic graphs, particularly the inclusion ideal graph in certain cases. While these findings offer a clearer understanding of the interplay between ring structure and graph properties, several natural questions remain open. These questions point to possible directions for further investigation, especially in extending the current results and exploring additional structural and numerical aspects of the graph.

\begin{openproblem}
The following problems remain open:
\begin{enumerate}
    \item What are the numerical graph parameters of $\Gamma'_{Reg}(R)$ when $R$ is a finite commutative regular ring? In particular, can these parameters be determined using the algorithm developed in this paper?
    \item What is the structure of the graph $\Gamma'_{Reg}(R)$ when $R$ is a finite non-commutative regular ring?
\end{enumerate}
\end{openproblem}

\section{Conclusion}
In this paper, we have investigated the structure and properties of the generalized von Neumann inverse graph $\Gamma_{Reg}(R)$ and its reduced form $\Gamma'_{Reg}(R)$ for finite commutative von Neumann regular rings. By focusing on the reduced graph, the study captures the essential structural behavior while avoiding trivial components. A central aspect of the analysis is the decomposition of such rings into finite direct products of fields, which provides a natural framework for describing the associated graphs.

A main contribution of this work is the development of a constructive approach to determine the structure of $\Gamma'_{Reg}(R)$ for an arbitrary finite commutative von Neumann regular ring. In particular, we present an explicit algorithm that allows the graph to be built systematically from simpler components, starting from the case of fields and extending through iterative constructions involving graph operations. This method shows that the structure of $\Gamma'_{Reg}(R)$ is strongly determined by the decomposition $R \cong F_1 \times F_2 \times \dots \times F_n$, and can be derived from the base graph associated with $(\mathbb{F}_2)^n$.

Various graph-theoretic properties of $\Gamma'_{Reg}(R)$ are characterized in terms of the algebraic structure of $R$. These include conditions for connectivity, acyclicity, and planarity, as well as precise descriptions of vertex degrees, girth, and the existence and number of pendant vertices. Moreover, specific graph classes such as paths, cycles, and wheels are completely classified in this setting. The presence of an induced subgraph isomorphic to the house graph is also shown to depend explicitly on the sizes of the component fields.

In addition to these results, we establish a direct connection between the generalized von Neumann inverse graph and the inclusion ideal graph. In the special case where $R \cong (\mathbb{F}_2)^n$, it is shown that the subgraph induced by $R \setminus \{0_R, 1_R\}$ in $\Gamma'_{Reg}(R)$ is isomorphic to $In(R)$, which leads to the conclusion that $\Gamma'_{Reg}(R) \cong K_1 + In(R)$. This relationship reveals that, in this setting, the structure of $\Gamma'_{Reg}(R)$ can be completely described in terms of the inclusion relations among ideals of the ring.

Overall, the results highlight a clear interplay between ring-theoretic properties and graph-theoretic structures. The graph $\Gamma'_{Reg}(R)$ not only reflects the decomposition of the ring but also encodes detailed algebraic information through its local and global features. The algorithmic framework developed in this work provides a practical tool for constructing and analyzing these graphs, while the connection with inclusion ideal graphs offers further insight and opens the way for future research on algebraic graphs arising from more general classes of rings.

\section*{Acknowledgement}
This research was carried out during the graduate study of the first author at Universitas Gadjah Mada (UGM), Yogyakarta, Indonesia, with support from the Department of Mathematics, UGM, and funding provided by the Ministry of Research, Technology, and Higher Education of the Republic of Indonesia (Kementerian Riset, Teknologi, dan Pendidikan Tinggi) through the PMDSU (Program Magister Menuju Doktor untuk Sarjana Unggul) Scholarship, 2024–2028, under Contract No. 074/C3/DT.05.00/PL-MULTITAHUN LANJUTAN/2026; 1059/UN1/DITLIT/Dit-Lit/PT.01.03/2026.

\end{document}